\newcommand*\figpath{./}

\documentclass[3p,times]{article}

\usepackage{graphicx}   
\usepackage[utf8]{inputenc} 
\usepackage[colorlinks=true, linkcolor=myred, citecolor=myred]{hyperref}         
\usepackage{listings}
\usepackage[outdir=./]{epstopdf}
\usepackage{amsmath,amsfonts,amssymb}
\usepackage{mathrsfs}
\usepackage{amsthm}
\usepackage{pdfpages}	
\usepackage{booktabs}	
\usepackage{times}            
\usepackage{colordvi}        
\usepackage{color}
\usepackage{xcolor}
\usepackage{multimedia}       
\usepackage{array,multirow,makecell}  
\usepackage{tabularx}
\usepackage{enumitem}
\usepackage{csvsimple}
\usepackage{comment}
\usepackage{float}
\usepackage{subcaption}
 \usepackage{lineno}

\usepackage{color}
\definecolor{myred}{rgb}{0.8,0.1,0.1}
\definecolor{mygreen}{rgb}{0.1,0.8,0.1}
\definecolor{darkgreen}{rgb}{0.1,0.4,0.1}
\definecolor{mygray}{gray}{0.93}

\usepackage{listingsutf8}

\lstdefinestyle{ffem}{
	framerule=0.6pt,
	texcl=true,
	basicstyle=\scriptsize\ttfamily\bfseries,
	keywordstyle=\color{blue}\bfseries,
	emphstyle=\color{blue}
	\bfseries,
	commentstyle=\color{darkgreen}\bfseries
	\itshape,
	stringstyle=\itshape,
	tabsize=4,
	frame=lines,
	numbers=none,
	numberstyle=\tiny,
	breaklines=true,
	showstringspaces=false,
	morekeywords={real,plot,border,mesh,label,buildmesh,adaptmesh,cmm,problem,int2d,int1d,fespace,func,string,on,dx,dy,cout,for,if,int,ofstream,ifstream},
	literate={?}{{\`a}}1 {�}{{\`e}}1 {�}{{\'e}}1,
	moredelim=[is][\color{blue}\footnotesize\bf]{|}{|},
	moredelim=[is][\color{darkgreen}\footnotesize\bf]{|+}{+|},
	morecomment=[l]{//},
	morestring=[s]{"}{"},
	morecomment=[s]{/*}{*/},
	escapeinside={@*}{*@},
	abovecaptionskip=0cm}

\usepackage[font=it,skip=\baselineskip]{caption}
\usepackage[normalem]{ulem}

\renewcommand{\vec}[1]{\boldsymbol{#1}}

\usepackage[hmargin={3cm,2.3cm},vmargin={3cm,3cm}]{geometry}

\newtheorem{theorem}{Theorem}

\newtheorem{corollary}{Corollary}
\newtheorem{definition}{Definition}
\newtheorem{proposition}{Proposition}
\newtheorem{remark}{Remark}
\newtheorem{assumption}{Assumption}
\def\LOGO{}
\font\fonteupmc=pagk at 10 true pt
\font\plutotgros=pagk at 12 true pt
\def\LAN{\plutotgros}
\def\upmc{\fonteupmc}

\def\entete{
	\vbox {
		\hbox to \hsize{%
			\hskip -1 cm\vbox{\LOGO}\hskip 0.45 true cm \vbox{\hbox{\LAN
					Laboratoire de math�matiques Raphael Salem} \vskip .2 true cm \hbox{\upmc
					Universit� de Rouen} \vskip .5 true cm
				\hbox{\upmc Avenue de l'Universit�, BP.12,
					76801 Saint-�tienne-du-Rouvray}}
			\hfill}
		{\vskip .9 true cm
			\hbox{\hskip -1 cm} \vskip .1 true cm
			\hbox{\hskip -1 cm}}}}

\def\polP{\mathbb P}

\def\iRe{\frac{1}{Re}}
\newcommand{\norm}[2][]{\lVert #2 \rVert_{#1}}
\newcommand{\nx}[1]{\norm[H^1]{#1}}
\newcommand{\nq}[1]{\norm[L^2]{#1}}

\title{On the convergence of a low order Lagrange finite element approach for natural convection problems}

\author{I. Danaila, C. Legrand, F. Luddens}
\date{Universit{\'e} de Rouen Normandie\\
Laboratoire de Math{\'e}matiques Rapha\"el Salem, CNRS, UMR 6085 \\
 Avenue de l'Universit\'e, BP. 12 \\
 76801 Saint-\'Etienne-du-Rouvray, France
	}

\begin{document}
\maketitle

\begin{abstract}
The purpose of this article is to study the convergence of a low order finite element approximation for a natural convection problem. We prove that the discretization based on $P_1$ polynomials for every variable (velocity, pressure and temperature) is well-posed if used with a  penalty term in the divergence equation, to compensate the loss of an inf-sup condition. With mild assumptions on the pressure regularity, we recover convergence for the Navier-Stokes-Boussinesq system, provided the penalty term is chosen in accordance with the mesh size. We express conditions to obtain optimal order of convergence. We illustrate theoretical convergence results with extensive examples. The computational cost that can be saved by this approach is also assessed.
\end{abstract}


\section{Introduction}

The finite element method was proved to be very effective for the numerical simulation of Phase Change Materials, described by Navier-Stokes equations supplemented with Boussinesq approximation for natural convection and an enthalpy model for the evolution of temperature. A common approach is to use a single domain method \cite{ArticleToolBox,Zimmerman} for the momentum equation. Projection schemes \cite{El-Haddad_etal_2022} or Newton based algorithms \cite{ArticleToolBox,ParallelToolBox} have been shown to give accurate results. Enthalpy models have been proven to be suitable for phase change, even with different thermophysical properties between the two phases, especially when combined with adaptive mesh refinement \cite{Belhamadia_etal_2012,Belhamadia_etal_2019}.

The use of Taylor-Hood finite element for the velocity-pressure unknowns ensures convergence of the method \cite{ErrorEstimateHecht}.  The stability of the scheme is dictated by the inf-sup condition. Indeed, when this condition is uniformly satisfied at the discrete level, the well-posedness of the underlying problem is ensured. As a result, the pressure does not exhibit spurious modes. A large number of finite elements have been proven to satisfy the inf-sup condition, for instance the classical Taylor-Hood element. The lowest equal-order finite element $P_1$-$P_1$ pair does not satisfy this condition: the finite element space for the velocity is not rich enough to control the spurious pressure mode. Apparition of unphysical pressure oscillations has been highlighted in \cite[section 5.2.5]{LivreGuermondErn}. Several methods have been introduced to neutralize spurious modes.  One of the commonly employed consist to enrich the velocity space in adding one degree of freedom per element, associated with the barycentre of the element. This new finite element, called $P_{1b}$, can be used for many different equations. In order to avoid the instability problem related to the choice of $P_1$ finite element for the velocity, several methods induce smaller discrete spaces for the pressure unknown and/or use an additional stabilization: projection stabilization \cite{HighOrderTermStabilization, ErrorAnalysisLocalProjectionStabilization}, the stabilization based on two local Gauss integrations \cite{LocalGaussStabilization, ErrorAnalysisNCZhang}, symmetric pressure stabilization \cite{SymetricPressureStabilization}, and many others. This paper focuses on the stabilized finite element method based on minimal pressure stabilization procedures introduced by \cite{LivreBrezziFortin}. The aim of the study is to estimate the effect of changing the discrete problem to a penalized problem and get corresponding a priori error estimates.

There exists a vast literature on error estimations for Navier-Stokes equations. Error estimates for the Stokes problem in a bounded smooth domain with slip boundary condition have been established~\cite{PenltyMethod}, using 
$P_1$-$P_1$ or $P_{1b}$-$P_1$ finite element approximations and an additional penalty method on the boundary conditions. Estimation of the optimal order of convergence for the velocity and the pressure for the time-dependent Stokes equation with discrete inf-sup stable virtual space for $k \geq 2$ and non-divergence free is analysed in~\cite{StabilizationMethod}. Several pressure stabilizations are studied in~\cite{ShenPseudoCompressibilitymethods}, with error analysis of time discretization schemes for corresponding Navier-Stokes equations.
 
The natural convection problem has been analysed using a Backward Euler and a fully Crank-Nicolson scheme with a variational multi scale method and a stabilized term. Both schemes were proven to be unconditionally stable. Error bounds were derived for a finite element space discretization satisfying the inf-sup condition~\cite{ErrorAnalysisNCZhang}.
The stability of a finite element approximation scheme with inf-sup condition for phase change problems has been assessed in~\cite{EstimErrorNC}. Assuming standard hypotheses on the discrete spaces, existence and stability of solutions of the Galerkin scheme associated to Navier-Stokes-Boussinesq equations can be obtained~\cite{ExistanceStabilityNSB}. Corresponding Cea’s estimate for smooth solutions can also be derived. The use of mini element $P_{1b}$-$P_1$ pair was analysed in~\cite{ErrorEstimateHecht}. A priori error estimates were derived for first and second order (in time) numerical schemes.
Error bound using $P_l$-$P_l$ pairs with ($l\geq1$) have been established with different pressure stabilization, for example Local Projection Stabilization \cite{ErrorAnalysisLocalProjectionStabilization} or grad-div pressure stabilization \cite{PenltyMethod}. 

This paper focuses on the approximation scheme introduced in~\cite{ArticleToolBox}. A constant penalty term is added to the discretized mass conservation equation. From an algebraic point of view, this penalty term avoids the use of pivoting by eliminating a null block in the discretization matrix. It also ensures that the discrete pressure has zero average. We let this penalty term vary according to the mesh size, to recover convergence for the velocity and temperature unknowns, even if the inf-sup condition is not satisfied, e.g. for the $P_1$-$P_1$ pair. Using similar technique to \cite{ErrorEstimateHecht}, we give a priori estimates and numerical illustration to support the use of such elements.

The paper is organized as follows: in \S\ref{sec-equationsdiscretization}, we present the finite element approximation for the natural convection. The main result consists in Theorem~\ref{thm_improved} and is presented in \S\ref{sec-main-result}, along with immediate corollaries illustrating the convergence of the numerical scheme. In \S\ref{sec-stokesproj}, we introduce a modified projection operator onto the finite element space, and establish some approximation results. This operator allows us to prove Theorem~\ref{thm_improved} in \S\ref{sec-proof}. Finally, numerical results for the projection operator as well as the natural convection problem are reported in \S\ref{sec-numerics}, in agreement with theoretical results.

\section{Framework and discretization}\label{sec-equationsdiscretization}

\subsection{Framework}

Throughout the paper, $\Omega$ denotes a smooth bounded domain in $\mathbb R^2$. To use known regularity results, we assume that its boundary $\partial\Omega$ is of class $C^{1,1}$. We are interested in solving the natural convection problem modelled by the Navier-Stokes equations supplemented with the Boussinesq approximation. 

Denoting by $\vec{u}$, $p$ and $\theta$ the dimensionless velocity of the fluid, pressure and temperature respectively, the Navier-Stokes-Boussinesq system of equations reads: 

\begin{eqnarray}
\nabla\cdot \vec{u}&=&0, \label{eq-div} \\ \vspace{0.2cm}
\frac{\partial \vec{u}}{\partial t} + {(\vec{u}\nabla) \vec{u}} +\nabla p -\frac{1}{Re}{\nabla^2 \vec{u}} 
- f_B(\theta)\, \vec{e}_y &=&\vec F, \label{eq-qmvt} \\ \vspace{0.2cm}
\frac{\partial \theta}{\partial t} + \nabla \cdot\left(  \theta \vec{u}\right) -
\nabla \cdot\left( \frac{1}{Re Pr} \nabla \theta \right)  &=& g. \label{eq-energ} 
\end{eqnarray}
$f_B$  stands for the Boussinesq buoyancy force, assumed to be a linear function of the temperature:
\begin{equation}\label{eq-fB}
f_B(\theta) = \frac{Ra}{Pr Re^2} \theta.
\end{equation}
We denote by $Ra$, $Re$ and $Pr$ the Rayleigh, Reynolds and Prandtl numbers respectively. Finally, $\vec{F}$ and $g$ denote external force and heat source.

The system \eqref{eq-div}-\eqref{eq-qmvt}-\eqref{eq-energ} is supplemented with Dirichlet boundary conditions  on $\vec u$ and $\theta$.

\subsection{Weak formulation}

To use finite element approximations, a weak formulation of the previous system is needed. We introduce classical Hilbert spaces : 
\begin{equation}
	 X=H^1_0(\Omega), \quad Q=\left\{q\in L^2(\Omega) \mid \int_{\Omega} q=0\right\}, \quad \vec X= X\times X,
\end{equation}
and define the following bilinear and trilinear forms: for $\vec u, \vec v, \vec w\in \vec X$, $p \in Q$ and $\theta,\psi\in X$, 

\begin{eqnarray}\label{operator_form_1}
a(\vec u,\vec v) &=& \left( \nabla \vec u, \nabla \vec v\right),\\
\bar a(\theta,\phi) &=& \left( \nabla\theta, \nabla \phi\right),\\
b(\vec u,p) &=& -\left( \nabla\cdot\vec u, p\right),\\
c(\vec u,\vec v,\vec w) &=& \frac 12 \left( \vec u\nabla \vec v, \vec w\right) - \frac 12 \left( \vec u\nabla \vec w, \vec v\right),\label{operator_c} \\ 
\bar c(\vec u,\theta, \psi) &=& \frac 12 \left( \vec u\nabla \theta,\vec  \psi\right)-\frac 12 \left( \vec u\nabla \psi,\vec  \theta\right) ,\label{operator_form_2}
\end{eqnarray}
where $(u,v)=\int_{\Omega}u\cdot v$ denotes the scalar product in $L^2(\Omega)$. If $\vec u\in \vec X$ is such that $\nabla\cdot\vec u=0$, operators \eqref{operator_c}-\eqref{operator_form_2} are equivalent to
\begin{align}
c(\vec u,\vec v,\vec w) &= \left( \vec u\nabla\vec v, \vec w\right) = \int_\Omega \left(\vec u\nabla \vec v\right)\cdot \vec w, \\
\bar c(\vec u,\theta,\psi) &= \left( \vec u\nabla\theta, \psi\right) = \int_\Omega \left(\vec u\nabla \theta\right)\psi.
\end{align}

If $(\vec u, p, \theta)$ is a solution to Eqns.~\eqref{eq-div}-\eqref{eq-qmvt}-\eqref{eq-energ} and $(\vec v,q,\psi)\in \vec X\times L^2\times X$, we obtain by integration by parts the following weak formulation: 
\begin{eqnarray}
b(\vec u, q) &=&0, \label{eq-div-weak} \\ \vspace{0.2cm}
\left(\frac{\partial \vec u}{\partial t}, \vec v\right) + \frac{1}{Re}a(\vec u, \vec v) + c(\vec u,\vec u,\vec v) + b(\vec v, p) - \left(f_B(\theta)\vec e_y,\vec v\right) &=& \left( \vec F, \vec v\right),\label{eq-qmvt-weak} \\ \vspace{0.2cm}
\left(\frac{\partial \theta}{\partial t},\psi\right) + \frac{1}{Re Pr}\bar a(\theta,\psi) + \bar c(\vec u,\theta,\psi)  &=& \left(g,\psi\right). \label{eq-energ-weak} 
\end{eqnarray}

\subsection{Properties of bilinear and trilinear forms}

From definitions~\eqref{operator_form_1}-\eqref{operator_form_2}, the following classical coercivity and continuity estimates hold: there exist some positive constants $\alpha, \bar\alpha, A, \bar A$ such that, $\forall \vec u,\vec v \in \vec X$, $\forall p \in Q$ and $\forall \theta, \phi \in X,$
\begin{eqnarray}
\lvert a(\vec u,\vec v) \rvert &\leq& A\,\lVert \vec u \rVert_{H^{1}}\,\lVert \vec v \rVert_{H^{1}},\\\label{a-coercive}
\lvert a(\vec v,\vec v) \rvert&\geq& \alpha\,\lVert \vec v \rVert_{H^{1}},\\
\lvert \bar a(\theta,\phi)\rvert&\leq& \bar A\,\lVert \theta \rVert_{H^{1}} \,\lVert \phi \rVert_{H^{1}},\\
\lvert \bar{a}(\theta,\theta) \rvert &\geq& \bar{\alpha}\, \lVert \theta \rVert_{H^{1}},\\
\lvert b(\vec u,p) \rvert &\leq& \lVert \vec u \rVert_{H^{1}}\,\lVert  p \rVert_{L^2}.
\end{eqnarray}

We also use the continuity of trilinear forms $c$ and $\bar c$ : there exist positive constants $C$ and $\bar C$ such that, $\forall \vec u,\vec v, \vec w \in \vec X$ and $\forall \theta,\psi \in X,$
\begin{eqnarray}
c(\vec u,\vec v,\vec w) &\leq& C\, \lVert \nabla \vec u \rVert_{L^2}\,\lVert \nabla \vec v \rVert_{L^2}\,\lVert \vec w \rVert_{L^2}^{\frac{1}{2}}\,\lVert \nabla \vec w \rVert_{L^2}^{\frac{1}{2}},\label{eq-c-continu}\\
\bar c(\vec u,\theta,\psi) &\leq& \bar C\, \lVert \nabla\vec u \rVert_{L^2}\, \lVert \nabla\theta \rVert_{L^2}\,\lVert \psi \rVert_{L^2}^{\frac{1}{2}}\,\lVert \nabla\psi \rVert_{L^2}^{\frac{1}{2}}\label{eq-barc-continu}.
\end{eqnarray}

Owing to skew-symmetry properties, $\vec v$ and $\vec w$ can be swapped in the right hand side of \eqref{eq-c-continu}, and similarly $\theta$ and $\psi$ in the right hand side of \eqref{eq-barc-continu}. In order to get the best possible convergence rate for the natural convection problem, we rely on the continuity result of the following Proposition~\ref{prop:continuity_c_improved} that holds in the two-dimensional case.

\begin{proposition}\label{prop:continuity_c_improved}
For any $\sigma>0$, there exists a positive constant $C$ depending only on $\Omega$ and $\sigma$ such that
\begin{eqnarray}
&\forall \vec u, \vec v, \vec w,\in \vec X,\qquad \left| c(\vec u,\vec v,\vec w)\right| \leq C \norm[H^\sigma]{\vec u}\nx{\vec v}\nx{\vec w}, \label{eq-c-uHsigma}\\
&\forall \vec u, \vec v, \vec w,\in \vec X,\qquad \left| c(\vec u,\vec v,\vec w)\right| \leq C \norm[H^\sigma]{\vec v}\nx{\vec u}\nx{\vec w}\label{eq-c-vHsigma}. 
 \end{eqnarray}
\end{proposition}

\begin{proof}
We start with $0<\sigma <1$. Note that if \eqref{eq-c-uHsigma} holds for $0<\sigma<1$, then it holds for any $\sigma'>\sigma$.

From Cauchy-Schwarz and Holder inequalities, we have
$$
\left|\int_\Omega \vec u\nabla \vec v\vec w\right| \leq \norm[L^2]{\nabla \vec v}\norm[L^2]{\vec u\vec w} \leq \norm[L^2]{\nabla \vec v}\norm[L^q]{\vec u}\norm[L^{q^*}]{\vec w},
$$
where $q>2$ is such that $\sigma=1-\frac 2q$  and $q^*$ is such that $\frac 1q+\frac 1{q^*}=\frac12$. The injections $L^q\subset H^{\sigma}$ and $L^{q^*}\subset H^1$ are continuous (see \cite{DiNezzaEtal12}), so that
$$
\left|\int_\Omega \vec u\nabla \vec v\vec w\right| \leq C \norm[L^2]{\nabla \vec v}\norm[H^\sigma]{\vec u}\nx{\vec w} \leq C \norm[H^\sigma]{\vec u}\nx{\vec v}\nx{\vec w}.
$$
From this inequality, we also obtain
$$
\left|\int_\Omega \vec u\nabla \vec w\vec v\right| \leq C \norm[H^\sigma]{\vec u}\nx{\vec v}\nx{\vec w}.
$$
Gathering the two inequalities, together with the definition of $c$ yields the desired result \eqref{eq-c-uHsigma}.

In order to prove \eqref{eq-c-vHsigma}, we use the fact that
$$
c(\vec u,\vec v,\vec w) = -\int_\Omega \vec u\nabla \vec w\vec v - \frac12\int_\Omega (\nabla \cdot \vec u)\vec v\vec w.
$$
We proceed as before : the first term is treated using $\norm[L^2]{\vec u\vec v}$, and the second one with $\norm[L^2]{\vec w\vec v}$.

\end{proof}

Finally, we recall that $\vec X\times Q$ satisfies the inf-sup condition (also called LBB condition, e.g. \cite{brezzi_1974}): there exists $\beta>0$ only depending on $\Omega$ such that, for all $p\in Q$,
\begin{equation}\label{eq:LBB_condition}
\beta \nq{p} \leq \sup_{\vec u\in\vec X\backslash\{0\}}\frac{ b(\vec u,p) }{\nx{\vec u}}.
\end{equation}

\subsection{Finite element approximation}

Let us introduce a family of uniform and regular meshes $\mathcal T_h$, indexed by $h$ (typically the size of a triangle). Given an integer $\ell$, and an element $K\in \mathcal{T}_h$, we denote by ${P}_{\ell}(K)$ the space of polynomials of degree less than, or equal to $\ell$, defined on $K$. We introduce the following finite element spaces 
\begin{eqnarray}
 X_h&:=&\{v_h\in C(\Omega)^d ;  v_h\mid_K\in {P}_{p_u}(K)^d, \quad \forall K \in \mathcal{T}_h\}\\
 Q_h&:=&\{q_h\in C(\Omega)^d ;  q_h\mid_K\in {P}_{p_p}(K)^d, \quad \forall K \in \mathcal{T}_h\}\\
 W_h&:=&\{w_h\in C(\Omega)^d ;  w_h\mid_K\in {P}_{p_\theta}(K)^d, \quad \forall K \in \mathcal{T}_h\}
\end{eqnarray}
where $p_u\geq1$, $p_p\geq 1$ and $p_\theta\geq1$. In the applications, we want to use $p_u=p_p=p_\theta=1$.

We are interested in the approximation of the natural convection problem for $t\in[0;t_f]$ where $t_f$ is a given final time. Let us choose $N_f$ a number of time steps, and define $\delta t = \frac{t_f}{N_f}$.  We denote by $t_n = n\delta t$ for any integer $n$ and $(\vec u_h^n,p_h^n,\theta_h^n)\in X_h\times Q_h\times W_h$ our approximation of $(\vec u^n,p^n,\theta^n)$, where $\vec u^n$ stands for $\vec u(t_n)$ (the same for $p^n,\theta^n$). Introducing the notations
$$
\delta_t \vec u_h^n = \frac{\vec u_h^{n}-\vec u_h^{n-1}}{\delta t}, \quad \delta_t \theta_h^n = \frac{\theta_h^{n}-\theta_h^{n-1}}{\delta t},
$$
the discretized system corresponding to~\eqref{eq-div-weak}-\eqref{eq-energ-weak} is: 

Find $(\vec{u}_h^{n+1},p_h^{n+1},\theta_h^{n+1})$ such that, for any $(\vec v_h,q_h,\psi_h)\in X_h\times Q_h\times W_h$,
\begin{eqnarray} 
b(\vec u_h^{n+1},q_h)-\gamma_h(p_h^{n+1},q_h)&=&0, \label{eq-div-discrete} \\ 
\left(\delta_t\vec u_h^{n+1},\vec v_h\right) +\frac{1}{Re}a(\vec{u}_h^{n+1},\vec v_h) + c(\vec u_h^{n+1},\vec u_h^{n+1},\vec v_h) &\nonumber \\+b( \vec v_h,p_h^{n+1})
-\left(f_B(\theta_h^{n+1})\vec{e}_y ,\vec v_h\right) &=& \left(\vec F^{n+1},\vec v_h\right) , \label{eq-qmvt-discrete} \\ 
\left(\delta_t\theta_h^{n+1},\psi_h\right) +
\frac{1}{Re Pr}\bar{a} (\theta_h^{n+1} , \psi_h)   + \bar{c}(\vec u_h^{n+1},\theta_h^{n+1},\psi_h)   &=& \left(g^{n+1},\psi_h\right).\label{eq-energ-discrete} \label{Variational_formulation_discrete} 
\end{eqnarray}

The resulting system is a non linear problem solved using a Newton algorithm, as in \cite{ArticleToolBox}. $\gamma_h$ is a positive constant that might depend on the mesh size. This penalty parameter is the key player in our analysis, and deserve some remarks:
\begin{itemize}
\item it is well-known that, using Taylor-Hood element for the velocity-pressure for example allows one to use $\gamma_h=0$. However, from a computational point of view, this leads to a invertible matrix with a zero diagonal block. Hence, it requires some pivoting. Letting $\gamma_h>0$ removes the necessity of using pivots.
\item since $Q_h$ is not a subspace of $Q$ (there is no constraint on the mean value of an element of $Q_h$), taking $\gamma_h>0$ together with the boundary conditions on $\vec u_h$ and the Stokes formula ensures that the pressure has zero mean value. In this respect, this term can be viewed as a \emph{penalty} term.
\item even if the couple $X_h\times Q_h$ does not satisfy an inf-sup condition, taking $\gamma_h>0$ will provide the well-posedness of the linearised system, allowing the use of the Lax-Milgram lemma. In this respect, it can also be viewed as a \emph{stabilization} term. See also details in \cite[Rem. 4.3, p. 67]{LivreGiraultRaviart}.
\item it is clear that using a constant value for $\gamma_h$ cannot lead to convergence when $h$ goes to zero. We will show however that, with a suitable choice of $\gamma_h$, we can obtain convergence for our problem.
\end{itemize}

Let us conclude this section with the introduction of suitable projection operators. We assume that there exists a family of operators $\mathcal C_h: L^2(\Omega)\to Q_h$ that satisfies the following properties: there exists $C$ independent of $h$ such that, for any $0\leqslant s < \frac 32$ , $s\leqslant r\leqslant 1+p_p$ and $q\in H^r(\Omega)$, the following inequality holds:
\begin{align}\label{eq:ClementP}
\lVert q - \mathcal C_h q\rVert_{H^s} &\leq Ch^{r-s}\lVert q \rVert_{H^r}.
\end{align}
As projection operators, one might think of the Clément interpolant, or the Scott-Zhang interpolant \cite{Scott-Zhang}. Abusing the notations, we also denote by $\mathcal C_h$ the $H^1$ projections from $\vec X$ to $X_h$ and from $X$ to $W_h$. For these operators, \eqref{eq:ClementP} is also satisfied since it can be inferred from their definitions
\begin{align}
a(\mathcal C_h\vec u,\vec v_h) &= a(\vec u,\vec v_h) \qquad \forall \vec u\in \vec X, \vec v_h\in X_h, \label{eq:ClementU} \\
\bar a(\mathcal C_h\theta,\psi_h) &= \bar a(\theta,\psi_h) \qquad \forall \theta\in X, \psi_h\in W_h. \label{eq:ClementT}
\end{align}

\section{Main convergence result}\label{sec-main-result}
\subsection{Notations and assumptions}
Let $0<\sigma<1$ (close to $0$) and define $\tilde \theta_h^{n}=\mathcal C_h\theta^n$. We define $(\vec{\tilde u}_h^n,\tilde p_h^n)=\polP_{\gamma_h,0}^{X_h,Q_h}(\vec u^n,p^n)$, that is: for any $\vec v_h\in X_h$ and any $q_h\in Q_h$,
$$
\frac1{Re}a(\vec{\tilde u}_h^n, \vec v_h) + b(\vec v_h, \tilde p_h^n) - b(\vec{\tilde u}_h^n, q_h) + \gamma_h(\tilde p_h^n, q_h) = \frac1{Re}a(\vec{ u}^n, \vec v_h) + b(\vec v_h, p^n) - b(\vec{u}^n, q_h).
$$
The approximations properties of this operator are detailed in Section~\ref{sec-stokesproj}. The errors are decomposed as:
\begin{align*}
\vec e_u^{n} &= \vec u^n - \vec u_h^n = \left( \vec u^n - \vec{\tilde u}_h^n\right) - \left( \vec u_h^n - \vec{\tilde u}_h^n\right)  = \eta_u^n - \varphi_u^n, \\
e_p^n &= p^n - p_h^n = \left( p^n - \tilde p_h^n\right) - \left( p_h^n-\tilde p_h^n\right) = \eta_p^n-\varphi_p^n, \\
e_\theta^n &= \theta^n - \theta_h^n = \left( \theta^n - \tilde \theta_h^n\right) - \left( \theta_h^n-\tilde \theta_h^n\right) = \eta_\theta^n-\varphi_\theta^n.
\end{align*} 
We also define $R_u^n$ and $R_\theta^n$ as :
$$
R_u^n = \partial_t \vec u^n - \delta_t \vec u^n,\quad R_\theta^n = \partial_t \theta^n - \delta_t \theta^n.
$$
We can use $(\vec v_h,q_h,\psi_h)\in X_h\times Q_h\times W_h$, as a test function in the weak formulation \eqref{eq-div-weak}-\eqref{eq-qmvt-weak}-\eqref{eq-energ-weak} at time $t_{n+1}$ so that:
\begin{eqnarray}
b(\vec u^{n+1}, q_h) &=&0, \label{eq-div-weak-disc} \\ \vspace{0.2cm}
\left(\delta_t \vec u^{n+1}, \vec v_h\right) + \frac{1}{Re}a(\vec u^{n+1}, \vec v_h) + c(\vec u^{n+1},\vec u^{n+1},\vec v_h)&\nonumber \\ + b(\vec v_h, p^{n+1}) - \left(f_B(\theta^{n+1})\vec e_y,\vec v_h\right) &=& \left( \vec F^{n+1}, \vec v_h\right) - \left( R_u^{n+1}, \vec v_h\right),\label{eq-qmvt-weak-disc} \\ \vspace{0.2cm}
\left(\delta_t \theta^{n+1},\psi_h\right) + \frac{1}{Re Pr}\bar a(\theta^{n+1},\psi_h) + \bar c(\vec u^{n+1},\theta^{n+1},\psi_h)  &=& \left(g^{n+1},\psi_h\right) - \left(R_\theta^{n+1},\psi_h\right). \label{eq-energ-weak-disc} 
\end{eqnarray} 

In this article, we focus on solutions that exhibit mild regularity properties. This is expressed by the following assumptions.

\begin{assumption}\label{AssStabilitySolution}
$\vec u$ and $\theta$ are bounded in $H^1$ and $p$ is bounded in $L^2$, and we introduce $\mathcal N>0$ such that
\begin{equation}\label{eq-def-N}
\forall t\in[0;t_f], \qquad \lVert \vec u(t)\rVert_{H^1} + \nq{p(t)} + \lVert \theta(t)\rVert_{H^1} \leq \mathcal N.
\end{equation}
\end{assumption}

\begin{assumption}\label{AssRegularitySolution-2}
There exists $0\leq s_\theta\leq p_\theta$ such that $\vec{u} \in L^{\infty}(0,t_f; H^{2})$, $\partial_{t}\vec{u} \in L^{\infty}(0,t_{f}; H^{2})$, $\partial_{tt}\vec{u} \in L^{\infty}(0,t_{f}; L^{2})$, $p\in L^{\infty}(0,t_{f}; H^{1})$, $\partial_t p\in L^{\infty}(0,t_{f}; H^{1})$, $\theta\in L^{\infty}(0,t_{f}; H^{1+s_{\theta}})$, $\partial_{t}\theta\in L^{\infty}(0,t_{f}; H^{s_{\theta}})$, $\partial_{tt}\theta\in L^{\infty}(0,t_{f}; L^{2})$. We then introduce
\begin{align}
 M_2(\vec u,p,\theta)&:= \norm[L^{\infty}(L^2)]{\partial_{tt}\vec u}^2 + \norm[L^{\infty}(L^2)]{\partial_{tt}\theta}^2 + \norm[L^{\infty}(H^{1+s_\theta})]{\theta}^2 \\
	&+ \norm[L^{\infty}(H^{s_\theta})]{\partial_t \theta}^2 
	+\norm[L^\infty(H^2)]{\vec u}^2 + \norm[L^\infty(H^1)]{p}^2 \\
	&+ \norm[L^{\infty}(L^2)]{p}^2
	+\norm[L^\infty(H^2)]{\vec u}^4 + \norm[L^\infty(H^1)]{p}^4 \\
	&+ \norm[L^{\infty}(L^2)]{p}^4
	+\norm[L^{\infty}(H^2)]{\partial_t\vec u}^2 + \norm[L^{\infty}(H^1)]{\partial_t p}^2 + \norm[L^{\infty}(L^2)]{\partial_t p}^2\label{eq-notation-nuptheta-2}
\end{align}
\end{assumption}

Note that Assumption~\ref{AssStabilitySolution} is redundant with Assumption~\ref{AssRegularitySolution-2}, but we keep it in order to separate terms that are bounded by $\mathcal N$ from those bounded by $M_2$. Throughout the paper, we will denote by $C$ a positive constant that is independent of $h$ and $M_2$. Unless stated otherwise, $C$ might depend on $\mathcal N$, $\Omega$, $\sigma$ and/or the dimensionless parameters $Re$, $Pr$ and $Ra$. Its value may change at every line.

\begin{theorem}\label{thm_improved}
	Let $(\vec{u}, p, \theta)$ be the solution of \eqref{eq-div}-\eqref{eq-qmvt}-\eqref{eq-energ} such that Assumptions~\ref{AssStabilitySolution} and \ref{AssRegularitySolution-2} hold. Let $0<\sigma<1$ be a small real number. Then the following error estimates holds for any $n\leqslant N_f$:
	\begin{align}
&	\lVert \varphi_u^n \rVert_0^2 +\lVert  \varphi_\theta^n \rVert_0^2  \leq C \Big(\delta t^2\norm[L^{\infty}(L^2)]{\partial_{tt}\vec u}^2 + \delta t^2\norm[L^{\infty}(L^2)]{\partial_{tt}\theta}^2 + h^{2s_\theta}\norm[L^{\infty}(H^{1+s_\theta})]{\theta}^2 \nonumber \\
	&+ h^{2s_\theta}\norm[L^{\infty}(H^{s_\theta})]{\partial_t \theta}^2 
	+\frac{h^{4-2\sigma}}{\gamma_h^{2-\sigma}}\norm[L^\infty(H^2)]{\vec u}^2 + \frac{h^{4-\sigma}}{\gamma_h^{2-\sigma}}\norm[L^\infty(H^1)]{p}^2 \nonumber \\
	&+ \gamma_h^2\norm[L^{\infty}(L^2)]{p}^2
	+\frac{h^4}{\gamma_h^2}\norm[L^\infty(H^2)]{\vec u}^4 + \frac{h^{4}}{\gamma_h^{2}}\norm[L^\infty(H^1)]{p}^4 \nonumber \\
	&+ \gamma_h^4\norm[L^{\infty}(L^2)]{p}^4
	+\frac{h^4}{\gamma_h^2}\norm[L^{\infty}(H^2)]{\partial_t\vec u}^2 + \frac{h^{4}}{\gamma_h^2}\norm[L^{\infty}(H^1)]{\partial_t p}^2 + \gamma_h^2\norm[L^{\infty}(L^2)]{\partial_t p}^2\Big) \label{eq_estimation_improved}
	\end{align}
	\end{theorem}
	
The proof of this theorem is postponed to Section \ref{sec-proof}. To get a convergence result, we need the following corollary to this result. 
	
\begin{corollary}\label{cor:improved}
Under the assumptions of Theorem \ref{thm_improved}, the following error estimate holds, for $h^2<\gamma_h<1$:
\begin{equation}\label{eq-improved}
\lVert \vec e_u^n \rVert_0^2 +\lVert  e_\theta^n \rVert_0^2 \leq C\left( \delta t^2 + h^{2s_\theta} + \frac{h^{4-2\sigma}}{\gamma_h^{2-\sigma}} + \gamma_h^2 \right)  M_2(\vec u,p,\theta).
\end{equation}
\end{corollary}

\begin{proof}
Using the triangular inequality, we obtain
$$
\lVert \vec e_u^n \rVert_0^2 +\lVert  e_\theta^n \rVert_0^2 \leq 2\lVert \varphi_u^n \rVert_0^2 +2\lVert  \varphi_\theta^n \rVert_0^2 + 2\lVert \eta_u^n \rVert_0^2 +2\lVert  \eta_\theta^n \rVert_0^2.
$$
The first two terms correspond to the left hand side of \eqref{eq_estimation_improved}, so that
$$
\lVert \varphi_u^n \rVert_0^2 +\lVert  \varphi_\theta^n \rVert_0^2 \leq \left( \delta t^2 + h^{2s_\theta} + \frac{h^{4-2\sigma}}{\gamma_h^{2-\sigma}} + \gamma_h^2 \right)  M_2(\vec u,p,\theta).
$$

For the last two terms, we use the properties of $\mathcal C_h$ and $\tilde{\vec u}_h$, see Eqns. \eqref{eq:ClementP} and \eqref{Norm_proj_L2_final}
\begin{align*}
\lVert  \eta_u^n \rVert_0^2 = \lVert  \vec u^n - \tilde {\vec u}_h^n \rVert_0^2 &\leq C\left( \frac{h^4}{\gamma_h^2} \norm[H^2]{\vec u^n}^2 + \frac{h^4}{\gamma_h}\norm[H^1]{p^n}^2 + \gamma_h^2\norm[L^2]{p^n}^2 \right) \\
&\leq C\left( \frac{h^4}{\gamma_h^2} + \frac{h^4}{\gamma_h} + \gamma_h^2\right) M_2(\vec u,p,\theta),
 \\
\lVert  \eta_\theta^n \rVert_0^2 = \lVert  \theta^n - \mathcal C_h \theta^n \rVert_0^2 &\leq C h^{2+2s_\theta}  \lVert  \theta^n\rVert_{H^{1+s_\theta}}^2 \leq C h^{2+2p_\theta} M_2(\vec u,p,\theta).
\end{align*}
Owing to the assumptions on $\gamma_h$, the latter right hand sides are bounded by the right hand side of \eqref{eq-improved}. 
\end{proof}

\begin{remark}
The convergence rate is not limited by $s_\theta$, therefore we can set $s_\theta=1$ (e.g. $p_\theta=1$ means that we also use $P_1$ elements for $\theta$). Then the estimate becomes
$$
\lVert \vec e_u^n \rVert_0^2 +\lVert  e_\theta^n \rVert_0^2 \leq C\left( \delta t^2 + h^{2} + \frac{h^{4-2\sigma}}{\gamma_h^{2-\sigma}} + \gamma_h^2 \right)  M_2(\vec u,p,\theta).
$$
If we choose $\gamma_h = h$, then we obtain
$$
\lVert \vec e_u^n \rVert_0^2 +\lVert  e_\theta^n \rVert_0^2 \leq C\left( \delta t^2 + h^{2-\sigma} \right)  M_2(\vec u,p,\theta)
$$
and we get an almost first order accuracy in $h$, supported by our numerical results, see Section~\ref{sec-numerics}.
\end{remark}

\begin{remark}
In Theorem~\ref{thm_improved} and its corollary, the constant $C$ depends on $\sigma$ and goes to infinity as $\sigma\to 0$. Hence, we do not recover the maximal order of convergence, but we can be as close as desired.
\end{remark}

\section{Modified Stokes projection}\label{sec-stokesproj}

It is possible to establish a convergence result, similar to Theorem~\ref{thm_improved}, by using $\mathcal C_h \vec{u}$ and $\mathcal C_h p$ instead of $\tilde{\vec u}_h$ and $\tilde p_h$. However, this leads to a predicted convergence order of $1/2$ in space, which is clearly not optimal. The purpose of this section is to design a modified Stokes projection that will help to improve the convergence rate. We call it \emph{modified Stokes projection} since, strictly speaking, it is not a projection.

\subsection{Definition of a family of operators}

\begin{definition}
For any $0\leq \lambda < 1$, we define the bilinear form $a_\lambda$ on $\vec X\times L^2$ by
$$
a_\lambda\big((\vec u, p),(\vec v,q)\big):= \iRe a(\vec u,\vec v) + b(\vec v,p) - b(\vec u,q) + \lambda(p,q).
$$
\end{definition}
We introduce a norm adapted to this bilinear form, viz.
$$
\norm[\lambda]{\vec u,p}:=\frac{1}{\sqrt{Re}}\nx{\vec u} + \sqrt{\lambda}\nq{p}.
$$
In the following, we will always assume that $0<\lambda<1$ and that $Re>1$.

\begin{proposition}
$a_\lambda$ is coercive and continuous with respect to the norm $\norm[\lambda]{\cdot}$. More precisely, there exist $C,C'$ independent of $\lambda$ and $Re$ such that, for any $(\vec u, p)\in \vec X\times L^2$ and $(\vec v, q)\in \vec X\times L^2$,
\begin{align}
\label{eq-alambda-coercive}
C' \norm[\lambda]{\vec u,p}^2 &\leq a_\lambda\big((\vec u, p),(\vec u,p)\big), \\
\label{eq-alambda-continuous}
\left| a_\lambda\big((\vec u, p),(\vec v,q)\big) \right| &\leq \frac{C\sqrt{Re}}{\sqrt\lambda}\norm[\lambda]{\vec u,p}\norm[\lambda]{\vec v,q}.
\end{align}
\end{proposition}
The coercivity of $a_\lambda$ directly derives from that of the bilinear form $a$, while the continuity is obtained from the continuity of $a$ and $b$ and Cauchy-Schwarz inequalities. Using this bilinear form, we define a family of operators in the following way:
\begin{definition}
Let $\mathcal X$ be a closed subset of $\vec X$ and $\mathcal Q$ be a closed subset of $L^2$. For any $0< \lambda < 1$ and $0\leq \gamma <1$, we define the operator $\polP_{\lambda,\gamma}^{\mathcal X,\mathcal Q}:\vec X\times L^2\to \mathcal X\times \mathcal Q$ in a weak form. For any $(\vec u,p)\in \vec X\times L^2$, $\polP_{\lambda,\gamma}^{\mathcal X,\mathcal Q}(\vec u,p) = (\vec{\tilde u},\tilde p)\in\mathcal X\times \mathcal Q$ is such that:
\begin{equation}\label{Definition_projection}
\forall (\vec v,q)\in\mathcal X\times \mathcal Q,\qquad a_\lambda\big((\vec{\tilde u}, \tilde p),(\vec v,q)\big) = a_\gamma\big((\vec u, p),(\vec v,q)\big).
\end{equation}
\end{definition}
Owing to the coercivity and continuity of $a_\lambda$, the operators $ \polP_{\lambda,\gamma}^{\mathcal X,\mathcal Q} $ are well-defined. We want to use the operator $\polP_{\gamma_h,0}^{X_h,Q_h}$ to establish convergence results for the natural convection problem. Note that, owing to \eqref{Definition_projection}, for $(\vec u,p)\in\vec X\times L^2$, if we define $(\vec u_h,p_h) = \polP_{\gamma_h,0}^{X_h,Q_h}(\vec u,p)$ and $(\vec{\tilde u},\tilde p) = \polP_{\gamma_h,0}^{\vec X,L^2}(\vec u,p)$, then we have
\begin{equation}
(\vec u_h,p_h) = \polP_{\gamma_h,\gamma_h}^{X_h,Q_h}(\vec{\tilde u},\tilde p).
\end{equation}

The modified Stokes projection is defined as $\polP_{\gamma_h,0}^{X_h,Q_h}$. Hence, its properties can be established through the properties of $\polP_{\gamma_h,\gamma_h}^{X_h,Q_h}$ and $\polP_{\gamma_h,0}^{\vec X,L^2}$. In order to study these two operators, we will repeatedly use Theorem 1.3 from \cite{Beirao97} which can be presented as

\begin{theorem}\label{thm_beirao}
Let $\vec f\in L^2$ and $g\in H^1\cap Q$. Let $(\vec u,p)\in\vec X\times Q$ such that
\begin{equation}
\begin{cases}
-\mu \Delta \vec u + \nabla p &= \vec f, \\
\nabla\cdot \vec u + \lambda p &= g,
\end{cases}
\end{equation}
where $\mu,\lambda$ are positive constants. Then $\vec u\in H^2$ and $p\in H^1$, and there exists a constant $C$ that depends only on $\Omega$ such that
\begin{equation}\label{eq-beirao}
\mu\lVert u\rVert_{H^2} + (1+\mu\lambda)\lVert p\rVert_{H^1} \leq C\left( \lVert \vec f\rVert_{L^2} + \mu\lVert g\rVert_{H^1}\right).
\end{equation}
\end{theorem}

\subsection{Analysis of the operators}

\subsubsection{The penalty operator $\polP_{\gamma_h,0}^{\vec X,L^2}$}
In this subsection, we establish convergence estimates for different norms.
\begin{proposition}\label{thm_proj_continu_first}
Let $\vec u\in\vec X$ and $p\in Q$. Let us introduce $(\vec{\tilde u},\tilde p) = \polP_{\gamma_h,0}^{\vec X,L^2}(\vec u,p)$. Then there exists a constant $C$ independent of $\gamma_h$ and $Re$ such that
\begin{equation}\label{eq-error-continuous-op-normeh}
\frac{1}{\sqrt{Re}}\nx{\vec u-\vec{\tilde u}} + \sqrt{Re}\nq{p-\tilde p} \leq \frac{C\gamma_h}{\sqrt{Re}}\nq{p}.
\end{equation}
\end{proposition}

\begin{proof}
From the definition of $\polP_{\gamma_h,0}^{\vec X,L^2}$, for any $(\vec v,q)\in\vec X\times L^2$, we infer that
$$
a_{\gamma_h}\big( (\vec{\tilde u},\tilde p),(\vec v,q)\big) = a_{\gamma_h}\big( (\vec{ u}, p),(\vec v,q)\big) - \gamma_h\left(p,q\right).
$$
Using the coercivity of $a_{\gamma_h}$, we obtain
\begin{align}
C'\norm[\gamma_h]{\vec u-\vec{\tilde u},p-\tilde p}^2& \leq  a_{\gamma_h}\big( (\vec{\tilde u}-\vec u,\tilde p-p),(\vec{\tilde u}-\vec u,\tilde p-p)\big) \nonumber \\
&\leq \gamma_h\left| \left(p, \tilde p-p\right)\right| \leq \gamma_h \nq{p}\nq{\tilde p-p}. \label{eq-proj-cont-1}
\end{align}
Instead of directly applying Young's inequality, we want to replace the norm of $(\tilde p-p)$ by the one of $(\vec{\tilde u}-\vec u)$. Using $(\vec v,0)$ as a test function in the definition of $\polP_{\gamma_h,0}^{\vec X,L^2}$ yields:
$$
\forall \vec v\in\vec X,\quad \iRe a(\vec{\tilde u}-\vec u, \vec v) + b(\vec v,\tilde p-p) = 0.
$$
Using the LBB condition \eqref{eq:LBB_condition}, we deduce that
\begin{equation}\label{eq:lbb_93_inter}
\beta \nq{\tilde p-p}\leq \sup_{\vec v\in\vec X\backslash\{0\}}\frac{b(\vec v,\tilde p-p)}{\nx{\vec v}} = \sup_{\vec v\in\vec X\backslash\{0\}}\frac{a(\vec u-\vec{\tilde u},\vec v)}{\nx{\vec v}} \leq \frac{A}{Re}\nx{\vec u-\vec{\tilde u}}.
\end{equation}
Using this inequality in \eqref{eq-proj-cont-1} yields
\begin{equation}
\norm[\gamma_h]{\vec u-\vec{\tilde u},p-\tilde p}^2 \leq \frac{C\gamma_h}{Re}\nq{p}\nx{\vec u-\vec{\tilde u}} \leq \frac{C\gamma_h}{\sqrt{Re}}\nq{p}\norm[\gamma_h]{\vec u-\vec{\tilde u},p-\tilde p},
\end{equation}
which in turn implies that
\begin{equation*}
\norm[\gamma_h]{\vec u-\vec{\tilde u},p-\tilde p} \leq \frac{C\gamma_h}{\sqrt{Re}}\nq{p}.
\end{equation*}
From this inequality, we obtain
\begin{equation}\label{eq:93_1}
\frac{1}{\sqrt{Re}}\nx{\vec u-\vec{\tilde u}} \leq \frac{C\gamma_h}{\sqrt{Re}}\nq{p}.
\end{equation}
Inserting this inequality~\eqref{eq:93_1} in \eqref{eq:lbb_93_inter} yields
\begin{equation}\label{eq:93_2}
\sqrt{Re} \nq{\tilde p-p}\leq \frac{A}{\beta\sqrt{Re}}\nx{\vec u-\vec{\tilde u}} \leq \frac{C\gamma_h}{\sqrt{Re}}\nq{p},
\end{equation}
and the desired result is finally obtained from \eqref{eq:93_1} and \eqref{eq:93_2}.
\end{proof}

\begin{proposition}\label{thm_proj_continu}
Let $\vec u\in\vec X\cap H^2$ and $p\in H^1\cap Q$. Let us introduce $(\vec{\tilde u},\tilde p) = \polP_{\gamma_h,0}^{\vec X,L^2}(\vec u,p)$. Then $\vec{\tilde u}\in H^2$, $\tilde p\in H^1\cap Q$ and the following estimates hold, for $C$ only depending on $\Omega$:
\begin{align}
\norm[H^2]{\vec{\tilde u}} &\leq C\left( \norm[H^2]{\vec u} + \gamma_h\norm[H^1]{p}\right), \label{eq-stability-continuous-op-u} \\
\norm[H^1]{\tilde p} &\leq C \norm[H^1]{p}, \label{eq-stability-continuous-op-p} \\
\iRe\norm[H^2]{\vec{\tilde u}-\vec u} + \left(1+\frac{\gamma_h}{Re}\right)\norm[H^1]{\tilde p-p} &\leq \frac{C\gamma_h}{Re}\norm[H^1]{p}.\label{eq-error-continuous-op}
\end{align}
\end{proposition}
\begin{proof}
We start by proving \eqref{eq-error-continuous-op}. Owing to the definition of $\polP_{\gamma_h,0}^{\vec X,L^2}$, we have 
\begin{equation}\label{eq-strong-continuous-op}
\begin{cases}
\displaystyle{-\iRe \Delta \left(\vec {\tilde u}-\vec u\right) + \nabla \left(\tilde p-p\right)} &= 0 , \\
\displaystyle{\nabla\cdot \left(\vec {\tilde u}-\vec u\right) + \gamma_h \left(\tilde p-p\right)} &= -\gamma_h p.
\end{cases}
\end{equation}
Applying Theorem~\ref{thm_beirao} with $\mu=\iRe$ and $\lambda=\gamma_h$ yields \eqref{eq-error-continuous-op}.

From this inequality, we infer that
\begin{align*}
\norm[H^2]{\vec{\tilde u}-\vec u} &\leq C\gamma_h \norm[H^1]{p}, \\
\norm[H^1]{\tilde p - p} &\leq C\norm[H^1]{p}.
\end{align*}
Then \eqref{eq-stability-continuous-op-u} and \eqref{eq-stability-continuous-op-p} are obtained using triangular inequalities.

\end{proof}

\subsubsection{The projection operator $\polP_{\gamma_h,\gamma_h}^{X_h,Q_h}$}
Let us now turn our attention to the projection operator $\polP_{\gamma_h,\gamma_h}^{X_h,Q_h}$.
\begin{proposition}\label{thm_proj_discrete}
Let $0\leqslant s\leqslant p_u$ and $0\leqslant s'\leqslant p_p$. Let $\vec u\in\vec X\cap H^{1+s}$ and $p\in H^{1+s'}\cap Q$. Let us introduce $(\vec{u}_h, p_h) = \polP_{\gamma_h,\gamma_h}^{X_h,Q_h}(\vec u,p)$.  Then  the following estimates hold, for $C$ depending only on $\Omega$:
\begin{align}
 \norm[\gamma_h]{\vec u-\vec u_h,p-p_h}^2 &\leq   \frac{C\; Re}{\gamma_h}\left( \frac{h^{2s}}{Re} \lVert \vec u\rVert_{H^{1+s}}^2 + h^{2(1+s')} \gamma_h\lVert p\rVert_{H^{1+s'}}^2\right), \label{eq:Norm_proj_H1} \\
 \lVert \vec u-\vec u_h \rVert_{L^2}^2&\leq   \frac{C\,Re^3}{\gamma_h^2}\left(\frac{h^{2(s+1)}}{Re} \lVert \vec u\rVert_{H^{1+s}}^2 + h^{4+2s'}\gamma_h \lVert p\rVert_{H^{1+s'}}^2\right).\label{eq:Norm_proj_L2}
 \end{align}
\end{proposition}

\begin{proof}
From the definition of the projection, for any $(\vec v_h,q_h)\in X_h\times Q_h$, we infer that
$$
a_{\gamma_h}\big((\vec u-\vec u_h,p-p_h),(\vec v_h,q_h)\big) = 0.
$$
Using this Galerkin orthogonality and the continuity of $a_{\gamma_h}$, we obtain
\begin{align}
&\ a_{\gamma_h}\big((\vec u-\vec u_h,p-p_h),(\vec u - \vec u_h,p-p_h)\big) = a_{\gamma_h}\big((\vec u-\vec u_h,p-p_h),(\vec u - \mathcal C_h\vec u,p-\mathcal C_h p)\big) \\
&\leq \frac{C\sqrt{Re}}{\sqrt{\gamma_h}}\norm[\gamma_h]{\vec u-\vec u_h,p-p_h}\norm[\gamma_h]{\vec u-\mathcal C_h\vec u,p-\mathcal C_h p}.
\end{align}
Using the definitions of the norm and $a_{\gamma_h}$, we also infer that
$$
\norm[\gamma_h]{\vec u-\vec u_h,p-p_h}^2 \leq C a_{\gamma_h}\big((\vec u-\vec u_h,p-p_h),(\vec u-\vec u_h,p-p_h)\big).
$$
As a result, we obtain
$$
\norm[\gamma_h]{\vec u-\vec u_h,p-p_h}^2 \leq \frac{C\sqrt{Re}}{\sqrt{\gamma_h}}\norm[\gamma_h]{\vec u-\vec u_h,p-p_h}\norm[\gamma_h]{\vec u-\mathcal C_h\vec u,p-\mathcal C_h p}.
$$
Using Young's inequality yields
$$
\norm[\gamma_h]{\vec u-\vec u_h,p-p_h}^2\leq \frac{C\, Re}{\gamma_h}\norm[\gamma_h]{\vec u-\mathcal C_h\vec u,p-\mathcal C_h p}^2.
$$
which in turn implies that:
\begin{equation}
\norm[\gamma_h]{\vec u-\vec u_h,p-p_h}^2\leq   \frac{2C\; Re}{\gamma_h}\left( \frac{1}{Re} \lVert \vec u-\mathcal C_h\vec u\rVert_{H^{1}}^2 + \gamma_h\lVert p-\mathcal C_h p \rVert_{L^2}^2\right).
\end{equation}
Using properties \eqref{eq:ClementP} of $\mathcal C_h$  yields \eqref{eq:Norm_proj_H1}.

The $L^2$ estimate is obtained using a Nitsche-Aubin argument. Let us define $\vec w\in \vec X$ and $r\in Q$ such that
\begin{equation}
\begin{cases}
\displaystyle{-\iRe \Delta \vec w + \nabla r} &= \vec u-\vec u_h, \\
\displaystyle{\nabla\cdot \vec w + \gamma_h r} &= 0.
\end{cases}
\end{equation}
Since $\vec u-\vec u_h\in L^2$, we can apply Theorem~\ref{thm_beirao} and we obtain $\vec w\in H^2$, $r\in H^1$ together with the estimate
\begin{equation}\label{eq-estimate-w-r}
\iRe\lVert w\rVert_{H^2} + \lVert r \rVert_{H^1} \leq C \lVert \vec u-\vec u_h\rVert_{L^2} .
\end{equation}
Owing to the definition of $\vec w$ and $r$, we have, for any $\vec v\in\vec X$ and $q\in L^2$
$$
a_{\gamma_h}\big((\vec w,r),(\vec v,q)\big) = \left(\vec u-\vec u_h, \vec v\right).
$$
Let us use this identity for $\vec v = \vec u-\vec u_h$ and $q=p_h-p$. Then we obtain
\begin{align*}
\norm[L^2]{\vec u-\vec u_h}^2 &= a_{\gamma_h}\big((\vec w,r),(\vec u-\vec u_h,p_h-p)\big) \\
&=  a_{\gamma_h}\big((\vec u-\vec u_h,p-p_h),(\vec w,-r)\big) \\
&= a_{\gamma_h}\big((\vec u-\vec u_h,p-p_h),(\vec w-\mathcal C_h w,\mathcal C_h r-r)\big).
\end{align*}
Owing to the continuity of $a_{\gamma_h}$, we deduce that
\begin{equation}\label{eq-l2-first}
\norm[L^2]{\vec u-\vec u_h}^2 \leq \frac{C\sqrt{Re}}{\sqrt{\gamma_h}}\norm[\gamma_h]{\vec u-\vec u_h,p-p_h}\norm[\gamma_h]{\vec w-\mathcal C_h\vec w,r-\mathcal C_h r}
\end{equation}
Let us estimate the term $\norm[\gamma_h]{\vec w-\mathcal C_h\vec w,r-\mathcal C_h r}$. From the definition of the norm and the properties of $\mathcal C_h$ (see \eqref{eq:ClementP}), we infer that
$$
\norm[\gamma_h]{\vec w-\mathcal C_h\vec w,r-\mathcal C_h r} \leq Ch\left( \frac{1}{\sqrt{Re}}\norm[H^2]{\vec w} + \sqrt{\gamma_h}\norm[H^1]{r}\right).
$$
Using the estimate \eqref{eq-estimate-w-r} and $Re>1$ leads to
$$
\norm[\gamma_h]{\vec w-\mathcal C_h\vec w,r-\mathcal C_h r} \leq Ch\sqrt{Re} \norm[L^2]{\vec u-\vec u_h}.
$$
Inserting this inequality in \eqref{eq-l2-first} yields
$$
\norm[L^2]{\vec u-\vec u_h} \leq \frac{C\,Re\,h}{\sqrt{\gamma_h}}\norm[\gamma_h]{\vec u-\vec u_h,p-p_h}.
$$
Taking the square of this inequality, together with \eqref{eq:Norm_proj_H1} yields \eqref{eq:Norm_proj_L2}.
\end{proof}

\begin{remark}
From \eqref{eq:Norm_proj_H1} and \eqref{eq:Norm_proj_L2}, one can note that there exists a constant $C$ independent of $h$, but depending on $\Omega$ and $Re$ such that, for any $\vec u\in \vec X\cap H^2$ and $p\in H^1\cap Q$, 
\begin{align*}
\norm[H^1]{\vec u-\vec u_h} \leq C\frac{h}{\sqrt{\gamma_h}}\left(\norm[H^2]{\vec u} + \norm[H^1]{p}\right), \\
\norm[L^2]{\vec u-\vec u_h} \leq C\frac{h^2}{\gamma_h}\left(\norm[H^2]{\vec u} + \norm[H^1]{p}\right).
\end{align*}
From interpolation theory, we obtain the estimate, for $0\leq s\leq 1$ : 
\begin{equation}\label{eq:proj-hs}
\norm[H^s]{\vec u-\vec u_h} \leq C\left(\frac{h}{\sqrt{\gamma_h}}\right)^{2-s}\left(\norm[H^2]{\vec u} + \norm[H^1]{p}\right).
\end{equation}
\end{remark}

\subsection{Error estimates for the operator $\polP_{\gamma_h,0}^{X_h,Q_h}$} 
 
We now have all the tools at hand to prove an approximation result for the operator that we want to use.

\begin{theorem}\label{Theorem_projection_0}
Let $\vec u\in\vec X\cap H^{2}$ and $p\in H^{1}\cap Q$. Let us introduce $(\vec{u}_h, p_h) = \polP_{\gamma_h,0}^{X_h,Q_h}(\vec u,p)$.  Then  the following estimates hold, for $C$ only depending on $\Omega$:
\begin{align} 
\norm[\gamma_h]{\vec u-\vec u_h, p-p_h}^2 &\leq C\left( \frac{h^2}{\gamma_h}\norm[H^2]{\vec u}^2 + h^2\gamma_h\norm[H^1]{p}^2 +  {Re\ h^2}\norm[H^1]{p}^2 + \frac{\gamma_h^2}{Re}\nq{p}^2\right), \label{Norm_proj_H1_final} \\
  \lVert \vec u-\vec u_h \rVert_{L^2}^2&\leq  C\left( \frac{Re^2 h^4}{\gamma_h^2}\norm[H^2]{\vec u}^2+\frac{Re^3 h^4}{\gamma_h}\norm[H^1]{p}^2 + \gamma_h^2\nq{p}^2\right).\label{Norm_proj_L2_final}
 \end{align}
\end{theorem}
 
\begin{proof}
Let us introduce $(\vec{\tilde u},\tilde p)=\polP_{\gamma_h,0}^{\vec X,L^2}(\vec u,p)$ so that we have $(\vec u_h,p_h)=\polP_{\gamma_h,\gamma_h}^{X_h,Q_h}(\vec{\tilde u},\tilde p)$. We start from triangular inequalities to get
$$
\norm[\gamma_h]{\vec u-\vec u_h, p-p_h}^2 \leq 2\norm[\gamma_h]{\vec u-\vec{\tilde u},p-\tilde p}^2 + 2\norm[\gamma_h]{\vec{\tilde u}-\vec u_h,\tilde p-p_h}^2.
$$
Let us now bound each term separately. From Proposition~\ref{thm_proj_continu_first}, we have the estimate
$$
\norm[\gamma_h]{\vec u-\vec{\tilde u},p-\tilde p}^2 \leq \frac{C\gamma_h^2}{Re}\nq{p}^2.
$$
For the second term, we use Proposition~\ref{thm_proj_discrete} with $s=1$ and $s'=0$ to obtain
$$
\norm[\gamma_h]{\vec{\tilde u}-\vec u_h,\tilde p-p_h}^2 \leq \frac{C Re\, h^2}{\gamma_h}\left( \frac{1}{Re}\norm[H^2]{\vec {\tilde u}}^2 + \gamma_h\norm[H^1]{\tilde p}^2\right).
$$
From \eqref{eq-stability-continuous-op-u}, we infer that
$$
\norm[\gamma_h]{\vec{\tilde u}-\vec u_h,\tilde p-p_h}^2 \leq \frac{C Re\, h^2}{\gamma_h}\left( \iRe\norm[H^2]{\vec u}^2 + \frac{\gamma_h^2}{Re}\norm[H^1]{p}^2 + \gamma_h\norm[H^1]{p}^2\right).
$$
Thus we have the desired estimate
$$
\norm[\gamma_h]{\vec u-\vec u_h, p-p_h}^2 \leq C\left( \frac{h^2}{\gamma_h}\norm[H^2]{\vec u}^2 + h^2\gamma_h\norm[H^1]{p}^2 +  {Re\ h^2}\norm[H^1]{p}^2 + \frac{\gamma_h^2}{Re}\nq{p}^2\right).
$$

We proceed similarly for the $L^2$ estimate, bounding $\norm[L^2]{\vec{\tilde u}-\vec u_h}$ using Proposition~\ref{thm_proj_discrete} and the error estimate \eqref{eq-error-continuous-op}, and bounding $\norm[L^2]{\vec{\tilde u}-\vec u}$ by $\sqrt{Re}\norm[\gamma_h]{ \vec{\tilde u}-\vec u,\tilde p-p}$ and using \eqref{eq-error-continuous-op-normeh}.
\end{proof} 
 
 \begin{remark}
Note that Eqns. ~\eqref{Norm_proj_H1_final}-\eqref{Norm_proj_L2_final} give a convergence result only if $\gamma_h\to 0$. The leading terms in \eqref{Norm_proj_L2_final} suggest to take $\gamma_h\sim Re^{1/2}h$, which would lead to order 1 in the $L^2$ norm, which is consistent with the results of the numerical simulations presented in Section~\ref{sec-numerics}.
 \end{remark}
 
 \begin{remark}
 We may need approximation results for the $H^s$ norm, for $0<s<1$. There exists a constant $C$ depending on $\Omega$ and $Re$, but not on $h$, such that
 \begin{equation}\label{eq:stokes-operator-hs}
 \norm[H^s]{\vec u-\vec u_h}^2 \leq C\left(\frac{h^{4-2s}}{{\gamma_h^{2-s}}}\left(\norm[H^2]{\vec u}^2 + \norm[H^1]{p}^2\right) + \gamma_h^2\nq{p}^2\right).
 \end{equation}
 This is obtained using the same decomposition as before, together with \eqref{eq:proj-hs}.
 \end{remark}
 
 Let us end this section with a stability result.
 \begin{proposition}\label{prop:stability_projection}
 Let $(\vec u,p)\in \vec X\times Q$ such that $\nabla\cdot\vec u = 0$. Let us introduce $(\vec u_h,p_h) = \polP_{\gamma_h,0}^{X_h,Q_h}(\vec u,p)$. There exists $C$ independent of $\vec u$ and $p$ such that
 \begin{equation}\label{eq:prop_stability_projection}
 \norm[H^1]{\vec u - \vec u_h} \leq C \left( \norm[H^1]{\vec u} + \norm[L^2]{p}\right).
 \end{equation}
 \end{proposition}
 
 \begin{proof}
 From the definition of $(\vec u_h,p_h)$ and the fact that $\nabla\cdot\vec u=0$, we infer that
 $$
 a_{\gamma_h}\left( (\vec u_h,p_h), (\vec u_h,p_h)\right) = a_0\left( (\vec u,p), (\vec u_h,p_h)\right) = \frac{1}{Re}a(\vec u,\vec u_h) + b(\vec u_h,p).
 $$
 Owing to the coercivity of $a_{\gamma_h}$ and the continuity of $a$ and $b$, we deduce that
 $$
 \frac{\alpha}{Re}\norm[H^1]{\vec u_h}^2 \leq C\norm[H^1]{\vec u_h}\left(\norm[H^1]{\vec u} + \norm[L^2]{p}\right).
 $$
 Using the triangular inequality yields
 $$
 \norm[H^1]{\vec u - \vec u_h} \leq \norm[H^1]{\vec u} + \norm[H^1]{\vec u_h} \leq C \left( \norm[H^1]{\vec u} + \norm[L^2]{p}\right).
 $$
 \end{proof}
\section{Proof of Theorem \ref{thm_improved}}\label{sec-proof}

We are now able to prove Theorem~\ref{thm_improved}, using the previously defined operator $\polP_{\gamma_h,0}^{X_h,Q_h}$. Owing to the definition of $\varphi_p^{n+1}$, $\varphi_u^{n+1},\eta_p^{n+1}$, $\eta_u^{n+1}$, summing \eqref{eq-div-discrete}-\eqref{eq-qmvt-discrete} and subtracting \eqref{eq-div-weak-disc}-\eqref{eq-qmvt-weak-disc} with corresponding test functions $\varphi_p^{n+1}$, $\varphi_u^{n+1}$ yields:

\begin{align}\nonumber
(\delta_t \varphi_u^{n+1},\varphi_u^{n+1}) &+\frac{1}{Re}a(\varphi_u^{n+1},\varphi_u^{n+1}) + \gamma_h(\varphi_p^{n+1},\varphi_p^{n+1}) \\
&=   (\delta_t \eta_u^{n+1},\varphi_u^{n+1}) +\frac{1}{Re}a(\eta_u^{n+1},\varphi_u^{n+1}) 
	+ b(\varphi_u^{n+1},\eta_p^{n+1}) \nonumber \\
	&+ c(\vec u^{n+1},\vec u^{n+1},\varphi_u^{n+1}) -  c(\vec u_h^{n+1},\vec u_h^{n+1},\varphi_u^{n+1})  
	 -	\frac{Ra}{Re^2Pr} (e_\theta^{n+1}\vec e_y,\varphi_u^{n+1}) \nonumber \\
	   &+\left(R_{ u}^{n+1}, \varphi_u^{n+1}\right)  - b(\eta_u^{n+1},\varphi_p^{n+1}) +  \gamma_h(\eta_p^{n+1},\varphi_p^{n+1})  - \gamma_h(p^{n+1},\varphi_p^{n+1}) \nonumber \\
	&=   (\delta_t \eta_u^{n+1},\varphi_u^{n+1}) + a_{\gamma_h}\big((\eta_u^{n+1},\eta_p^{n+1}),(\varphi_u^{n+1},\varphi_p^{n+1})\big)
 \nonumber \\
	&+ c(\vec u^{n+1},\vec u^{n+1},\varphi_u^{n+1}) -  c(\vec u_h^{n+1},\vec u_h^{n+1},\varphi_u^{n+1})  
	 -	\frac{Ra}{Re^2Pr} (e_\theta^{n+1}\vec e_y,\varphi_u^{n+1}) \nonumber \\
	   &+\left(R_{ u}^{n+1}, \varphi_u^{n+1}\right)  - \gamma_h(p^{n+1},\varphi_p^{n+1})\label{eq-div-qmvt-improved}
\end{align}
Owing to the definition of $\vec{\tilde u}_h^{n+1}$ and $\tilde p_h^{n+1}$, we obtain
\begin{align*}
&a_{\gamma_h}\big((\eta_u^{n+1},\eta_p^{n+1}),(\varphi_u^{n+1},\varphi_p^{n+1})\big) - \gamma_h(p^{n+1},\varphi_p^{n+1}) \\
&= a_{0}\big((\vec u^{n+1},p^{n+1}),(\varphi_u^{n+1},\varphi_p^{n+1})\big) - a_{\gamma_h}\big((\vec{\tilde u}_h^{n+1},\tilde p_h^{n+1}),(\varphi_u^{n+1},\varphi_p^{n+1})\big)  = 0.
\end{align*}
Thus, \eqref{eq-div-qmvt-improved} simply reads
\begin{align}\nonumber
(\delta_t \varphi_u^{n+1},\varphi_u^{n+1}) &+\frac{1}{Re}a(\varphi_u^{n+1},\varphi_u^{n+1}) + \gamma_h(\varphi_p^{n+1},\varphi_p^{n+1}) \\
&=(\delta_t \eta_u^{n+1},\varphi_u^{n+1})+\left(R_{ u}^{n+1}, \varphi_u^{n+1}\right)  -	\frac{Ra}{Re^2Pr} (e_\theta^{n+1}\vec e_y,\varphi_u^{n+1})
 \nonumber \\
	&+ c(\vec u^{n+1},\vec u^{n+1},\varphi_u^{n+1}) -  c(\vec u_h^{n+1},\vec u_h^{n+1},\varphi_u^{n+1})  . \label{eq-div-qmvt-simple}
\end{align}

The left hand side of \eqref{eq-div-qmvt-simple} is bounded from below by
$$
 \frac{1}{2\delta t}\left( \|\varphi_u^{n+1}\|_{L^2}^2 - \|\varphi_u^{n}\|_{L^2}^2 + \|\varphi_u^{n+1}-\varphi_u^n\|_{L^2}^2\right) + \frac{\alpha}{Re}\|\varphi_u^{n+1}\|_{H^1}^2 + \gamma_h\|\varphi_p^{n+1}\|_{L^2}^2
$$
so we are left with the right hand side to bound from above. Using Young's inequality, we obtain
	\begin{align}
	\left| (\delta_t \eta_u^{n+1},\varphi_u^{n+1}) \right| &\leq  \lVert \delta_t \eta_u^{n+1}\rVert_{L^2} \lVert \varphi_u^{n+1}\rVert_{L^2} \leq \frac{1}{2}\lVert \delta_t \eta_u^{n+1}\rVert_{L^2}^2 + \frac12\lVert \varphi_u^{n+1}\rVert_{L^2}^2\label{bound_eq_u_dttheta} \\
		\left| (R_{\vec u}^{n+1},\varphi_u^{n+1}) \right| &\leq  \lVert R_{\vec u}^{n+1}\rVert_{L^2} \lVert \varphi_u^{n+1}\rVert_{L^2} \leq \frac{1}{2}\lVert R_{\vec u}^{n+1}\rVert_{L^2}^2 + \frac12\lVert \varphi_u^{n+1}\rVert_{L^2}^2, \label{bound_eq_u_end}
	\end{align}
		where $\kappa_{u}$ and $\kappa_p$ are some positive constants, independent of $h$, yet to be chosen.
		
	The Boussinesq term is treated using the decomposition of $e_\theta^{n+1}$ : 
	\begin{align}
\left| (e_\theta^{n+1}\vec e_y,\varphi_u^{n+1})\right| &\leq  \|e_\theta^{n+1}\|_{L^2} \|\varphi_u^{n+1}\|_{L^2} \leq \left( \|\eta_\theta^{n+1}\|_{L^2}+\|\varphi_\theta^{n+1}\|_{L^2}\right) \|\varphi_u^{n+1}\|_{L^2}\nonumber \\
&\leq \frac{1}{2}\|\eta_\theta^{n+1}\|_{L^2}^2 + \frac{1}{2}\|\varphi_\theta^{n+1}\|_{L^2}^2 +\|\varphi_u^{n+1}\|_{L^2}^2 \label{bound_boussinesq}
	\end{align}

The trilinear terms are treated using
	\begin{eqnarray}\nonumber
	&&c(\vec u^{n+1},\vec u^{n+1},\varphi_u^{n+1}) - c(\vec u_h^{n+1},\vec u_h^{n+1},\varphi_u^{n+1})\\ \nonumber
	&=& c(\vec u^{n+1},\eta_u^{n+1},\varphi_u^{n+1}) + c(\eta_u^{n+1}, \vec u^{n+1},\varphi_u^{n+1}) - c(\eta_u^{n+1},\eta_u^{n+1},\varphi_u^{n+1}) \\ &-&c( \varphi_u^{n+1}, \vec u^{n+1},\varphi_u^{n+1}) + c(\varphi_u^{n+1},\eta_u^{n+1},\varphi_u^{n+1})\label{trilinear_u}
	\end{eqnarray}
The first three terms are handled using  \eqref{eq-c-uHsigma}-\eqref{eq-c-vHsigma} and a Young's inequality, for $0<\sigma<1$, and the last two are treated with \eqref{eq-c-continu} and the Young's inequality $xy\leq \frac{x^4}{4} + \frac{3y^{4/3}}{4}$:
\begin{align}
\left| c(\vec u^{n+1},\eta_u^{n+1},\varphi_u^{n+1})\right|  &\leq C\lVert \vec u^{n+1} \rVert_{H^{1}} \lVert \eta_u^{n+1} \rVert_{H^\sigma} \lVert \varphi_u^{n+1} \rVert_{H^{1}}\nonumber \\
&\leq \frac{C^2}{4\kappa_u}\lVert \vec u^{n+1} \rVert_{H^{1}}^2 \lVert \eta_u^{n+1} \rVert_{H^\sigma} ^2 + \kappa_u \lVert \varphi_u^{n+1} \rVert_{H^{1}}^2, \label{bound_eq_u_c_1} \\
\left| c(\eta_u^{n+1}, \vec u^{n+1},\varphi_u^{n+1})\right| &\leq C\lVert \vec u^{n+1} \rVert_{H^{1}} \lVert \eta_u^{n+1} \rVert_{H^\sigma} \lVert \varphi_u^{n+1} \rVert_{H^{1}}\nonumber \\
&\leq \frac{C^2}{4\kappa_u}\lVert \vec u^{n+1} \rVert_{H^{1}}^2 \lVert \eta_u^{n+1} \rVert_{H^\sigma} ^2 + \kappa_u \lVert \varphi_u^{n+1} \rVert_{H^{1}}^2, \\
\left| c(\eta_u^{n+1},\eta_u^{n+1},\varphi_u^{n+1})\right| &\leq C\lVert \eta_u^{n+1} \rVert_{H^{1}}^2  \lVert \varphi_u^{n+1} \rVert_{H^{1}} \nonumber \\
&\leq \frac{C^2}{4\kappa_u} \lVert \eta_u^{n+1} \rVert_{H^1} ^4 + \kappa_u \lVert \varphi_u^{n+1} \rVert_{H^{1}}^2, \label{bound_eq_u_c_3}  \\
\left| c( \varphi_u^{n+1}, \vec u^{n+1},\varphi_u^{n+1})\right| &\leq C\lVert \varphi_u^{n+1} \rVert_{H^1} \lVert \vec u^{n+1} \rVert_{H^1}  \lVert \varphi_u^{n+1} \rVert_{H^{1}}^{1/2}\lVert \varphi_u^{n+1} \rVert_{L^2}^{1/2}  \nonumber \\
&\leq C\lVert \vec u^{n+1} \rVert_{H^1} \lVert \varphi_u^{n+1} \rVert_{H^1}^{3/2} \lVert \varphi_u^{n+1} \rVert_{L^2}^{1/2}  \nonumber \\
&\leq \frac{C^4}{4\kappa_u^3}\lVert \vec u^{n+1} \rVert_{H^1}^4\lVert \varphi_u^{n+1} \rVert_{L^2}^{2} + \frac{3\kappa_u}{4} \lVert \varphi_u^{n+1} \rVert_{H^1}^{2}, \\
\left| c( \varphi_u^{n+1}, \eta_u^{n+1},\varphi_u^{n+1})\right| &\leq C\lVert \varphi_u^{n+1} \rVert_{H^1} \lVert \eta_u^{n+1} \rVert_{H^1}  \lVert \varphi_u^{n+1} \rVert_{H^{1}}^{1/2}\lVert \varphi_u^{n+1} \rVert_{L^2}^{1/2}  \nonumber \\
&\leq C\lVert \eta_u^{n+1} \rVert_{H^1} \lVert \varphi_u^{n+1} \rVert_{H^1}^{3/2} \lVert \varphi_u^{n+1} \rVert_{L^2}^{1/2} \nonumber \\
&\leq \frac{C^4}{4\kappa_u^3}\mathcal N^4 \lVert \varphi_u^{n+1} \rVert_{L^2}^{2} + \frac{3\kappa_u}{4} \lVert \varphi_u^{n+1} \rVert_{H^1}^{2}. \label{bound_eq_u_c_end}
\end{align}

Note that for the last term, we used \eqref{eq:prop_stability_projection} to bound $\norm[H^1]{\eta_u^{n+1}}$.  Gathering all these estimates leads to 
\begin{align}
&\ \ \ \ \ \ \ \frac{1}{2\delta t}\left( \|\varphi_u^{n+1}\|_{L^2}^2 - \|\varphi_u^{n}\|_{L^2}^2 + \|\varphi_u^{n+1}-\varphi_u^n\|_{L^2}^2\right) + \frac{\alpha}{Re}\|\varphi_u^{n+1}\|_{H^1}^2 + \gamma_h\|\varphi_p^{n+1}\|_{L^2}^2 \nonumber \\
&\leq \frac{1}{2}\lVert R_{\vec u}^{n+1}\rVert_{L^2}^2 + \frac{1}{2}\lVert \delta_t \eta_u^{n+1}\rVert_{L^2}^2 + \frac{Re}{2Re^2Pr}\|\eta_\theta^{n+1}\|_{L^2}^2+ \frac{Ra}{2Re^2Pr}\|\varphi_\theta^{n+1}\|_{L^2}^2  \\
&  + \frac{9\kappa_u}2\lVert \varphi_u^{n+1}\rVert_{H^1}^2+ C'_u \|\varphi_u^{n+1}\|_{L^2}^2  + \frac{C^2}{4\kappa_u} \lVert \eta_u^{n+1} \rVert_{H^{1}}^4 + \frac{C^2}{2\kappa_u}\nx{\vec u^{n+1}} \norm[H^{\sigma}]{\vec u^{n+1}}^2,
\end{align}
where $C'_u = 1+\frac{Ra}{Re^2Pr}+ \frac{C^4}{2\kappa_u^3} \mathcal N^4$.

Choosing $\kappa_u=\frac{\alpha}{9Re}$, there exists $C$ independent of $h$, such that
\begin{align}
\frac{1}{2\delta t}&\left( \|\varphi_u^{n+1}\|_{L^2}^2 - \|\varphi_u^{n}\|_{L^2}^2 + \|\varphi_u^{n+1}-\varphi_u^n\|_{L^2}^2\right) + \frac{\alpha}{2Re}\|\varphi_u^{n+1}\|_{H^1}^2 + \frac{\gamma}{2}\|\varphi_p^{n+1}\|_{L^2}^2 \nonumber \\
&\leq C\left( \lVert R_{\vec u}^{n+1}\rVert_{L^2}^2 + \lVert \delta_t \eta_u^{n+1}\rVert_{L^2}^2 + \|\eta_\theta^{n+1}\|_{L^2}^2 + \lVert \eta_u^{n+1} \rVert_{H^{\sigma}}^2 + \lVert \eta_u^{n+1} \rVert_{H^{1}}^4\right. \nonumber \\
& + \|\varphi_\theta^{n+1}\|_{L^2}^2 +  \|\varphi_u^{n+1}\|_{L^2}^2\big). \label{eq_estimate_u_improved}
\end{align}

We proceed similarly to get an estimate on the temperature, based on the approximation properties of $\mathcal C_h$.
\begin{align}
\frac{1}{2\delta t}&\left( \|\varphi_\theta^{n+1}\|_{L^2}^2 - \|\varphi_\theta^{n}\|_{L^2}^2 + \|\varphi_\theta^{n+1}-\varphi_\theta^n\|_{L^2}^2\right) + \frac{\bar\alpha K}{2Re\;Pr}\|\varphi_\theta^{n+1}\|_{H^1}^2 \nonumber \\
&\leq C\left( \lVert R_\theta^{n+1}\rVert_{L^2}^2 + \lVert\delta_t \eta_\theta^{n+1}\rVert_{L^2}^2 + \lVert  \eta_\theta^{n+1}\rVert_{H^1}^2 + \lVert  \varphi_\theta^{n+1}\rVert_{L^2}^2\right).
 \label{eq_estimate_T_improved}
\end{align}
Summing the last two estimates and removing some positive terms from the left hand side yields:
\begin{align}\label{eq:natconv_stokes_before_gronwall}
\frac{1}{2\delta t}&\left( \|\varphi_u^{n+1}\|_{L^2}^2 - \|\varphi_u^{n}\|_{L^2}^2 + \|\varphi_\theta^{n+1}\|_{L^2}^2 - \|\varphi_\theta^{n}\|_{L^2}^2\right) \\
&\leq C\Big( \lVert R_{\vec u}^{n+1}\rVert_{L^2}^2 + \lVert \delta_t \eta_u^{n+1}\rVert_{L^2}^2 + \|\eta_\theta^{n+1}\|_{L^2}^2 + \lVert \eta_u^{n+1} \rVert_{H^{\sigma}}^2 + \lVert \eta_u^{n+1} \rVert_{H^{1}}^4  \nonumber \\
&+  \lVert R_\theta^{n+1}\rVert_{L^2}^2 + \lVert\delta_t \eta_\theta^{n+1}\rVert_{L^2}^2 + \lVert  \eta_\theta^{n+1}\rVert_{H^1}^2 + \|\varphi_u^{n+1}\|_{L^2}^2 +\lVert  \varphi_\theta^{n+1}\rVert_{L^2}^2\Big).\nonumber
\end{align}

Before applying Gronwall lemma, we want to bound all the terms in the right hand side (except for the last two), using the properties of $\mathcal C_h\theta$, $\vec{\tilde u}_h$, $\tilde p_h$. We first notice that
\begin{align*}
\lVert R_{\vec u}^{n+1}\rVert_{L^2}& \leq C\delta t \lVert \partial_{tt}\vec{u}\rVert_{L^{\infty}(L^2)}, & \lVert R_{\theta}^{n+1}\rVert_{L^2}& \leq C\delta t \lVert \partial_{tt}\theta\rVert_{L^{\infty}(L^2)}, \\
\norm[L^2]{\eta_\theta^{n+1}}&\leq \norm[H^1]{\eta_\theta^{n+1}}\leq Ch^{s_\theta}\norm[L^{\infty}(H^{1+s_\theta})]{\theta}, & 
	\lVert\delta_t \eta_\theta^{n+1}\rVert_{L^2} & \leq Ch^{s_\theta} \lVert \partial_{t}\theta\rVert_{L^{\infty}(H^{s_\theta})}.
\end{align*}

Then, using the linearity of the operator $\polP_{\gamma_h,0}^{X_h,Q_h}$ together with estimates \eqref{eq:stokes-operator-hs} and \eqref{Norm_proj_H1_final}, we infer that
\begin{align*}
\norm[H^\sigma]{\eta_u^{n+1}}^2 &\leq C\left( \frac{h^{4-2\sigma}}{\gamma_h^{2-\sigma}}\norm[L^\infty(H^2)]{\vec u}^2 + \frac{h^{4-\sigma}}{\gamma_h^{2-\sigma}}\norm[L^\infty(H^1)]{p}^2 + \gamma_h^2\norm[L^{\infty}(L^2)]{p}^2\right), \\
\norm[H^1]{\eta_u^{n+1}}^4 &\leq C\left( \frac{h^4}{\gamma_h^2}\norm[L^\infty(H^2)]{\vec u}^4 + \frac{h^{4}}{\gamma_h^{2}}\norm[L^\infty(H^1)]{p}^4 + \gamma_h^4\norm[L^{\infty}(L^2)]{p}^4\right), \\
\norm[L^2]{\delta_t\eta_u^{n+1}}^2 &\leq C\left( \frac{h^2}{\gamma_h}\norm[H^2]{\delta_t\vec u^{n+1}}^2 + \frac{h^{2}}{\gamma_h}\norm[H^1]{\delta_t p^{n+1}}^2 + \gamma_h^2\norm[L^2]{\delta_t p^{n+1}}^2\right), \\
&\leq C\left( \frac{h^2}{\gamma_h}\norm[L^{\infty}(H^2)]{\partial_t\vec u}^2 + \frac{h^{2}}{\gamma_h}\norm[L^{\infty}(H^1)]{\partial_t p}^2 + \gamma_h^2\norm[L^{\infty}(L^2)]{\partial_t p}^2\right).
\end{align*}
Owing to these estimates, we introduce the new quantity 
\begin{align*}
	E(\vec u,p,\theta,\delta t,h,\gamma) &:=  \delta t^2\norm[L^{\infty}(L^2)]{\partial_{tt}\vec u}^2 + \delta t^2\norm[L^{\infty}(L^2)]{\partial_{tt}\theta}^2 + h^{2s_\theta}\norm[L^{\infty}(H^{1+s_\theta})]{\theta}^2 \\
	&+ h^{2s_\theta}\norm[L^{\infty}(H^{s_\theta})]{\partial_t \theta}^2 
	+\frac{h^{4-2\sigma}}{\gamma_h^{2-\sigma}}\norm[L^\infty(H^2)]{\vec u}^2 + \frac{h^{4-\sigma}}{\gamma_h^{2-\sigma}}\norm[L^\infty(H^1)]{p}^2 \\
	&+ \gamma_h^2\norm[L^{\infty}(L^2)]{p}^2
	+\frac{h^4}{\gamma_h^2}\norm[L^\infty(H^2)]{\vec u}^4 + \frac{h^{4}}{\gamma_h^{2}}\norm[L^\infty(H^1)]{p}^4 \\
	&+ \gamma_h^4\norm[L^{\infty}(L^2)]{p}^4
	+\frac{h^2}{\gamma_h}\norm[L^{\infty}(H^2)]{\partial_t\vec u}^2 + \frac{h^{2}}{\gamma_h}\norm[L^{\infty}(H^1)]{\partial_t p}^2 + \gamma_h^2\norm[L^{\infty}(L^2)]{\partial_t p}^2,
\end{align*}
so that \eqref{eq:natconv_stokes_before_gronwall} yields
	\begin{align}
	&\frac{1}{\delta t}\left(\lVert\varphi_u^{n+1}\rVert_{L^2}^2-\lVert\varphi_u^{n}\rVert_{L^2}^2  +\lVert\varphi_\theta^{n+1}\rVert_{L^2}^2-\lVert\varphi_\theta^{n}\rVert_{L^2}^2 \right)  \\
	&\leq  C\left( E(\vec u,p,\theta,\delta t,h,\gamma)+  \|\varphi_\theta^{n+1}\|_{L^2}^2 +  \|\varphi_u^{n+1}\|_{L^2}^2 \right). \label{eq_before_gronwall_ncstokes}
	\end{align}

Finally, we use the discrete Gronwall lemma, and we obtain the desired estimate, provided we choose $\vec u_h^0$ and $\theta_h^0$ such that $\varphi_u^0=0$ and $\varphi_\theta^0=0$.

\section{Numerical results}\label{sec-numerics}

In this section, we present numerical results to support theoretical estimates. It is worth mentioning that we performed the tests on domains that are not of class $C^{1,1}$. We start by illustrating \eqref{Norm_proj_H1_final} and \eqref{Norm_proj_L2_final} before turning our attention to the natural convection problem. We also illustrate the gain in computational cost allowed by the use of the linear element instead of the standard Taylor-Hood elements. All simulations are performed using the open-source software FreeFem++ \cite{hecht-2012-JNM,freefem}.

\subsection{The modified projection $\polP_{\gamma_h,0}^{X_h,Q_h}$}

Given $(\vec u,p)$, we compute $(\vec{\tilde u_h},\tilde p_h)\in X_h\times  Q_h$ such that 
\begin{equation}\label{Variational_formulation_projecteur}
\forall (\vec v_h,q_h)\in X_h\times Q_h,\qquad a_{\gamma_h}\big((\vec{\tilde u_h}, \tilde p_h),(\vec v_h,q_h)\big) = a_0\big((\vec u, p),(\vec v_h,q_h)\big).
\end{equation}

We want to illustrate the results of Theorem~\ref{Theorem_projection_0}.
The results for the $P_1$-$P_1$ elements (i.e. $p_u=p_p=1$) are better than expected. Therefore, we also performed simulations using $P_1$-$P_0$ elements, in order to illustrate the sharpness of our error estimates \eqref{Norm_proj_H1_final}-\eqref{Norm_proj_L2_final}. The choice of $\gamma_h$ modifies the estimates and consequently the order of convergence. Different values of $\gamma_h$ are tested to reach the best estimate. We present some numerical results for the following values of $\gamma_h$: \begin{itemize}
\item $\gamma_h=10^{-7}$, which is not supposed to yield any convergence, neither for the velocity nor the pressure;
\item $\gamma_h = Re^{1/3}h^{2/3}$, which is the best choice for the convergence of velocity in the $H^1$ norm: this should yield convergence of order $2/3$ for the velocity ($L^2$ and $H^1$ norms) and order $1/3$ for the pressure;
\item $\gamma_h = Re^{1/2}h$, which is the best choice for the convergence of velocity in the $L^2$ norm: this should yield convergence for the velocity of order $1$ in the $L^2$ norm, order $1/2$ in the $H^1$ norm, but no convergence in pressure;
\item $\gamma_h = Re\,h^2$, which is not supposed to yield any convergence, and even explosion for the pressure.
\end{itemize}

\subsubsection{Steady Burggraf flow (case \texttt{MP-Bur})}\label{Steady Burggraf flow}

We first focus on the Burggraf manufactured solution, see~\cite{ArticleToolBox}. This case is a time-independent recirculating flow inside a square cavity $[0 , 1] \times [ 0 , 1]$. It is similar to the well-known  lid driven cavity flow, but the velocity singularity at the top corners of the cavity is avoided. 
The exact solution of the flow is:
\begin{eqnarray}
\label{burggraf-u}
u_1(x,y) &=& \chi g'(x) h'(y), \\\nonumber
u_2(x,y) &=& - \chi g''(x) h(y), \\\nonumber
p(x,y)&=&\tilde{p}(x,y)-\frac{1}{|\Omega|}\int_\Omega\tilde{p}(x,y),
\end{eqnarray}
with $\chi >0$ is a scaling parameter and functions and the function $\tilde p$, $g$ and $h$ are defined by:
\begin{eqnarray}
\tilde{p}(x,y)   &=& \frac{\chi}{Re} \left( h^{(3)}(y) g(x) + g''(x)h'(y) \right) + \frac{\chi^2}{2} g'(x)^2 \left( h(y)h''(y)-h'(y)^2 \right)\\\nonumber
g(x) &=& \frac{x^5}{5} - \frac{x^4}{2} + \frac{x^3}{3}, \\ \nonumber
h(y) &=& y^4 - y^2. \\ \nonumber
\end{eqnarray}
Note that  the velocity at the top border of the cavity is:
\begin{equation}
u_1(x,1) = 2\chi (x^4 - 2x^3 + x^2), \quad u_2(x,1) =0,
\end{equation}
which ensures the continuity of the velocity at the corners ($\vec{u}(0,1) =\vec{u}(1,1) =0$), since non-slip walls are imposed for the other borders:
$\vec{u}(x,0) = \vec{u}(0,y) = \vec{u}(1,y) = 0$. The Reynolds number $Re$ is taken equal to 1 and $\chi$ equal to 8.

In all the tables, the column "rate" reports the computed order of convergence based on the computed errors. Empty cells mean that the iterative solver did not converge.

\begin{table}[H]
	\begin{center}
		\begin{tabular}{|c||c|c||c|c||c|c||c|c|c|}
			\hline
			\multirow{2}*{}& \multicolumn{2}{c||}{$\gamma_h\sim10^{-7}$} &\multicolumn{2}{c||}{$\gamma_h\sim Re^{1/3}h^{2/3}$}  &\multicolumn{2}{c||}{$\gamma_h \sim Re^{1/2}h$} &\multicolumn{2}{c|}{$\gamma_h \sim Re\,h^2$} \\
			\cline{1-9}
			h & \text{err} & \text{rate} & \text{err} & \text{rate} & \text{err} & \text{rate} & \text{err} & \text{rate}\\
			\hline
			\hline
			1/20  & 4.42E-03 &      & 5.27E-02& & 2.32E-02 &      & 3.92E-03 &      \\ \hline
			1/40  & 1.14E-03 & 1.95 &3.63E-02 & 0.54 & 1.21E-02 & 0.94 & 1.01E-03 & 1.96 \\ \hline
			1/80  & 2.93E-04 & 1.96 & 2.44E-02 & 0.57& 6.25E-03 & 0.96 & 2.58E-04 & 1.97 \\ \hline
			1/160 & 7.44E-05 & 1.98 & 1.61E-02 & 0.60& 3.17E-03 & 0.98 & 6.55E-05 & 1.98 \\ \hline
			1/200 & 4.78E-05 & 1.98 &1.40E-02&0.62& 2.54E-03 & 0.99 & 4.21E-05 & 1.98 \\ \hline
		\end{tabular}
	\end{center}%
	\caption{Case \texttt{MP-Bur} : $\lVert \vec u-\vec u_h \rVert_{L^2}$ for $P_1$-$P_1$ elements.}
	\label{tab:UL2P1P1_Burggraf}
\end{table}

\begin{table}[H]
	\begin{center}
		\begin{tabular}{|c||c|c||c|c||c|c||c|c|c|}
			\hline
			\multirow{2}*{}& \multicolumn{2}{c||}{$\gamma_h\sim 10^{-7}$} &\multicolumn{2}{c||}{$\gamma_h\sim Re^{1/3}h^{2/3}$}  &\multicolumn{2}{c||}{$\gamma_h\sim Re^{1/2}h$} &\multicolumn{2}{c|}{$\gamma_h\sim Re\,h^2$} \\
			\cline{1-9}
			h & \text{err} & \text{rate} & \text{err} & \text{rate} & \text{err} & \text{rate} & \text{err} & \text{rate}\\
			\hline
			\hline
			1/20  & 2.09E+05 &      &4.38E-01& & 5.01E-01 &      & 8.45E+00 &       \\ \hline
			1/40  & 5.49E+04 & 1.93 &2.97E-01& 0.56& 2.82E-01 & 0.83 & 8.91E+00 & -0.08 \\ \hline
			1/80  & 1.41E+04 & 1.96 &2.00E-01& 0.57 & 1.56E-01 & 0.85 & 9.15E+00 & -0.04 \\ \hline
			1/160 & 3.56E+03 & 1.98 & 1.32E-01& 0.59 & 8.57E-02 & 0.86 & 9.27E+00 & -0.02 \\ \hline
			1/200 & 2.29E+03 & 1.99 & 1.16E-01 & 0.61 & 7.08E-02 & 0.86 & 9.30E+00 & -0.01 \\ \hline
		\end{tabular}
	\end{center}
	\caption{Case \texttt{MP-Bur}  : $\lVert \vec p-\vec p_h \rVert_{L^2}$ for $P_1$-$P_1$ element}
	\label{tab:PL2P1P1_Burggraf}
\end{table}

\begin{table}[H]
	\begin{center}
		\begin{tabular}{|c||c|c||c|c||c|c||c|c|c|}
			\hline
			\multirow{2}*{}& \multicolumn{2}{c||}{$\gamma_h\sim 10^{-7}$} &\multicolumn{2}{c||}{$\gamma_h\sim Re^{1/3}h^{2/3}$}  &\multicolumn{2}{c||}{$\gamma_h\sim Re^{1/2}h$} &\multicolumn{2}{c|}{$\gamma_h\sim Re\,h^2$} \\
			\cline{1-9}
			h & \text{err} & \text{rate} & \text{err} & \text{rate} & \text{err} & \text{rate} & \text{err} & \text{rate}\\
			\hline
			\hline
			1/20  & 3.11E-01 &      & 3.57E-01&& 2.79E-01 &      & 2.85E-01 &      \\ \hline						
			1/40  & 1.59E-01 & 0.97 &2.16E-01 &0.72 & 1.41E-01 & 0.98 & 1.45E-01 & 0.97 \\ \hline
			1/80  & 8.09E-02 & 0.98 &1.35E-01 &0.68 & 7.09E-02 & 0.99 & 7.37E-02 & 0.98 \\ \hline
			1/160 & 4.08E-02 & 0.99 &8.48E-02& 0.67& 3.55E-02 & 1.00 & 3.71E-02 & 0.99 \\ \hline
			1/200 & 3.27E-02 & 0.99 & 7.32E-02 & 0.66& 2.84E-02 & 1.00 & 2.97E-02 & 0.99 \\ \hline
		\end{tabular}
	\end{center}
	\caption{Case \texttt{MP-Bur}  : $\lVert \vec u-\vec u_h \rVert_{H^1}$ for $P_1$-$P_1$ element}
	\label{tab:UH1P1P1_Burggraf}
\end{table}

We note from Tables~\ref{tab:UL2P1P1_Burggraf}-\ref{tab:PL2P1P1_Burggraf}-\ref{tab:UH1P1P1_Burggraf} that ($P_1$-$P_1$) elements give the expected behaviour for $\gamma_h\sim Re^{1/3}h^{2/3}$ and  $\gamma_h\sim Re^{1/2}h$. Nevertheless, we find a better performance than expected for $\gamma_h\sim 10^{-7}$ and $\gamma_h\sim Re\;h^2$.  However, the predicted behaviour from Theorem~\ref{thm_proj_continu_first} is recovered with ($P_1$-$P_0$) elements, see Tables~\ref{tab:UL2P1P0_Burggraf}-\ref{tab:PL2P1P0_Burggraf}-\ref{tab:UH1P1P0_Burggraf}.

\begin{table}[H]
	\begin{center}
		\begin{tabular}{|c||c|c||c|c||c|c||c|c|c|}
			\hline
			\multirow{2}*{}& \multicolumn{2}{c||}{$\gamma_h\sim 10^{-7}$} &\multicolumn{2}{c||}{$\gamma_h\sim Re^{1/3}h^{2/3}$}  &\multicolumn{2}{c||}{$\gamma_h\sim Re^{1/2}h$} &\multicolumn{2}{c|}{$\gamma_h\sim Re\,h^2$} \\
			\cline{1-9}
			h & \text{err} & \text{rate} & \text{err} & \text{rate} & \text{err} & \text{rate} & \text{err} & \text{rate}\\
			\hline
			\hline
			1/20  & 1.52E-01 &       &5.39E-02& & 2.88E-02 &      & 9.09E-02 &      \\ \hline
			1/40  & 1.53E-01 & -0.01 &3.66E-02&0.56 & 1.45E-02 & 0.99 & 9.10E-02 & 0.00 \\ \hline
			1/80  & 1.53E-01 & 0.00  & 2.45E-02 & 0.58 & 7.30E-03 & 0.99 & 9.11E-02 & 0.00 \\ \hline
			1/160 & 1.53E-01 & 0.00  & 1.61E-02 & 0.60 & 3.67E-03 & 0.99 & 9.11E-02 & 0.00 \\ \hline
			1/200 &          &       &1.40E-02&0.62& 2.94E-03 & 1.00 &          &      \\ \hline
		\end{tabular}
	\end{center}
	\caption{Case \texttt{MP-Bur}  : $\lVert \vec u-\vec u_h \rVert_{L^2}$ for $P_1$-$P_0$ element}
	\label{tab:UL2P1P0_Burggraf}
\end{table}

\begin{table}[H]
	\begin{center}
		\begin{tabular}{|c||c|c||c|c||c|c||c|c|c|}
			\hline
			\multirow{2}*{}& \multicolumn{2}{c||}{$\gamma_h\sim 10^{-7}$} &\multicolumn{2}{c||}{$\gamma_h\sim Re^{1/3}h^{2/3}$}  &\multicolumn{2}{c||}{$\gamma_h\sim Re^{1/2}h$} &\multicolumn{2}{c|}{$\gamma_h\sim Re\,h^2$} \\
			\cline{1-9}
			h & \text{err} & \text{rate} & \text{err} & \text{rate} & \text{err} & \text{rate} & \text{err} & \text{rate}\\
			\hline
			\hline
			1/20  & 1.27E+06 &      &1.22E+00 & & 3.06E+00 &       & 5.26E+01 &       \\ \hline
			1/40  & 6.39E+05 & 0.99 & 9.76E-01 & 0.32 & 3.15E+00 & -0.04 & 1.06E+02 & -1.01 \\ \hline
			1/80  & 3.20E+05 & 1.00 & 7.72E-01 & 0.34 & 3.19E+00 & -0.02 & 2.12E+02 & -1.00 \\ \hline
			1/160 & 1.60E+05 & 1.00 &6.09E-01 & 0.34 & 3.22E+00 & -0.01 & 4.24E+02 & -1.00 \\ \hline
			1/200 &          &     &5.64E-01& 0.34 & 3.22E+00 & -0.01 &          &       \\ \hline
		\end{tabular}
	\end{center}
	\caption{Case \texttt{MP-Bur}  : $\lVert \vec p-\vec p_h \rVert_{L^2}$ for $P_1$-$P_0$ element}
	\label{tab:PL2P1P0_Burggraf}
\end{table}

\begin{table}[H]
	\begin{center}
		\begin{tabular}{|c||c|c||c|c||c|c||c|c|c|}
			\hline
			\multirow{2}*{}& \multicolumn{2}{c||}{$\gamma_h\sim 10^{-7}$} &\multicolumn{2}{c||}{$\gamma_h\sim Re^{1/3}h^{2/3}$}  &\multicolumn{2}{c||}{$\gamma_h\sim Re^{1/2}h$} &\multicolumn{2}{c|}{$\gamma_h\sim Re\,h^2$} \\
			\cline{1-9}
			h & \text{err} & \text{rate} & \text{err} & \text{rate} & \text{err} & \text{rate} & \text{err} & \text{rate}\\
			\hline
			\hline
			1/20  & 1.46E+00 &      & 3.60E-01 & & 3.07E-01 &      & 9.04E-01 &      \\ \hline
			1/40  & 1.47E+00 & 0.00 &2.17E-01 & 0.73 & 1.58E-01 & 0.96 & 8.93E-01 & 0.02 \\ \hline
			1/80  & 1.47E+00 & 0.00 & 1.35E-01 & 0;69 & 8.04E-02 & 0.98 & 8.91E-01 & 0.00 \\ \hline
			1/160 & 1.46E+00 & 0.00 &8.49E-02 &0.67 & 4.05E-02 & 0.99 & 8.90E-01 & 0.00 \\ \hline
			1/200 &          &      & 7.32E-02 & 0.66 & 3.24E-02 & 0.99 &          &      \\ \hline
		\end{tabular}
	\end{center}
	\caption{Case \texttt{MP-Bur}  : $\lVert \vec u-\vec u_h \rVert_{H^1}$ for $P_1$-$P_0$ element}
	\label{tab:UH1P1P0_Burggraf}
\end{table}

\subsubsection{Steady natural convection  (case \texttt{MP-NC})}\label{Steady natural convection}

The first test case \texttt{MP-Bur} is an academic validation, on a regular manufactured solution. We also assess the accuracy of the modified Stokes projection on a more realistic case.

To do so, we consider the classical problem of the thermally driven square cavity $[ 0 , 1] \times [ 0 , 1]$, filled with air, described by the system of equations \eqref{eq-div}-\eqref{eq-energ}. The left wall is kept at a constant hot temperature $\theta_h = 0.5$ and the right wall is kept at a constant cold temperature $\theta_c = -0.5$. Top and bottom walls are adiabatic. Natural convection flow is computed for the Rayleigh number $Ra = 10^4$. The Prandtl number is set to $Pr = 0.71$ and the Reynolds number to $Re=\sqrt{\frac{Ra}{Pr}}$.

We compute a reference solution with ($P_2$-$P_1$-$P_2$) finite elements for the velocity, pressure and temperature on a fine fixed grid with mesh size $h=1/500$. $\gamma_h$ is taken equal to $0$. The non linear system \eqref{eq-div-discrete}-\eqref{eq-energ-discrete} is solved using a Newton algorithm. The initial state consists of motionless air ($\vec{u}=0$),  with a linear distribution of the temperature. For the time integration, we use a second-order Gear (BDF2) scheme and set $\delta t=h$. It has been proven in \cite{ErrorEstimateHecht} that this choice ensures convergence. Therefore, we use this result as reference solution.

The computations are performed until a steady state is reached. Then we use the computed $(\vec u,p)$ in the right hand side of \eqref{Variational_formulation_projecteur}.  The results are similar to the results for Case \texttt{MP-Bur} and still in agreement with theoretical estimates. Since the results are similar, we only present the $L^2$ error on the velocity and the $L^2$ error on the pressure in Tables~\ref{tab:UL2P1P1_NC}-\ref{tab:PL2P1P1_NC} ($P_1$-$P_1$) and Tables~\ref{tab:UL2P1P0_NC}-\ref{tab:PL2P1P0_NC} ($P_1$-$P_0$).

\begin{table}[H]
	\begin{center}
		\begin{tabular}{|c||c|c||c|c||c|c||c|c|c|}
			\hline
			\multirow{2}*{}& \multicolumn{2}{c||}{$\gamma_h\sim 10^{-7}$} &\multicolumn{2}{c||}{$\gamma_h\sim Re^{1/3}h^{2/3}$}
			&\multicolumn{2}{c||}{$\gamma_h\sim Re^{1/2}h$} &\multicolumn{2}{c|}{$\gamma_h\sim Re\,h^2$} \\
			\cline{1-9}
			h & err & rate & err & rate & err & rate & err & rate\\
			\hline
			\hline
			1/20  & 6.86E-03 &      & 4.19E-02 &      & 3.51E-02 && 6.31E-03 &      \\ \hline
			1/40  & 1.72E-03 & 2.00 & 2.77E-02 & 0.59 & 1.86E-02 & 0.92 & 1.60E-03 & 1.98 \\ \hline
			1/80  & 4.28E-04 & 2.00 & 1.81E-02 & 0.62 & 9.56E-03 & 0.96 & 4.01E-04 & 2.00 \\ \hline
			1/160 & 1.07E-04 & 2.00 & 1.16E-02 & 0.64 & 4.85E-03 & 0.98 & 1.00E-04 & 2.00 \\ \hline
			1/320 &  &  & 7.42E-03 & 0.65 & 2.44E-03 & 0.99 &  &  \\ \hline
		\end{tabular}
	\end{center}
	\caption{Case \texttt{MP-NC}  : $\lVert \vec u-\vec u_h \rVert_{L^2}$ for $P_1$-$P_1$-element}\label{tab:UL2P1P1_NC}
\end{table}

\begin{table}[H]
	\begin{center}
		\begin{tabular}{|c||c|c||c|c||c|c||c|c|c|}
			\hline
			\multirow{2}*{}& \multicolumn{2}{c||}{$\gamma_h\sim 10^{-7}$} &\multicolumn{2}{c||}{$\gamma_h\sim Re^{1/3}h^{2/3}$}
			&\multicolumn{2}{c||}{$\gamma_h\sim Re^{1/2}h$} &\multicolumn{2}{c|}{$\gamma_h\sim Re\,h^2$} \\
			\cline{1-9}
			h & err & rate & err & rate & err & rate & err & rate\\
			\hline
			\hline
			1/20  & 7.33E+00 &       & 6.94E-03 &      & 6.53E-03 &          & 1.43E-02 &      \\ \hline
			1/40  & 2.31E+00 & 1.66  & 4.15E-03 & 0.74 & 3.50E-03 & 0.90 & 1.36E-02 & 0.08 \\ \hline
			1/80  & 1.25E+00 & 0.89  & 2.51E-03 & 0.73 & 1.84E-03 & 0.93& 1.33E-02 & 0.03 \\ \hline
			1/160 & 1.25E+00 & 0.00  & 1.54E-03 & 0.71 & 9.58E-04 & 0.94 & 1.31E-02 & 0.01 \\ \hline
			1/320 &  &  & 9.49E-04 & 0.69 & 5.01E-04 & 0.93 & & \\ \hline
		\end{tabular}
	\end{center}
	\caption{Case \texttt{MP-NC} : $\lVert \vec p-\vec p_h \rVert_{L^2}$ for $P_1$-$P_1$-element}
	\label{tab:PL2P1P1_NC}
\end{table}

\begin{table}[H]
	\begin{center}
		\begin{tabular}{|c||c|c||c|c||c|c||c|c|c|}
			\hline
			\multirow{2}*{}& \multicolumn{2}{c||}{$\gamma_h\sim 10^{-7}$} &\multicolumn{2}{c||}{$\gamma_h\sim Re^{1/3}h^{2/3}$}
			&\multicolumn{2}{c||}{$\gamma_h\sim Re^{1/2}h$} &\multicolumn{2}{c|}{$\gamma_h\sim Re\,h^2$} \\
			\cline{1-9}
			h & err & rate & err & rate & err & rate & err & rate\\
			\hline
			\hline
			1/20  & 1.35E-01 &      & 4.27E-02 &      & 3.71E-02 &        & 8.78E-02 &       \\ \hline
			1/40  & 1.35E-01 & 0.00 & 2.80E-02 & 0.61 & 1.99E-02 & 0.90 & 8.78E-02 & 0.00  \\ \hline
			1/80  & 1.35E-01 & 0.00 & 1.81E-02 & 0.63 & 1.03E-02 & 0.95 & 8.78E-02 & 0.00  \\ \hline
			1/160 & 1.35E-01 & 0.00 & 1.16E-02 & 0.64 & 5.23E-03 & 0.97 & 8.78E-02 & 0.00  \\ \hline
			1/320 & & & 7.43E-03 & 0.65 & 2.64E-03 & 0.99 &  &  \\ \hline
		\end{tabular}
	\end{center}
	\caption{Case \texttt{MP-NC} : $\lVert \vec u-\vec u_h \rVert_{L^2}$ for $P_1$-$P_0$-element}\label{tab:UL2P1P0_NC}
\end{table}

\begin{table}[H]
	\begin{center}
		\begin{tabular}{|c||c|c||c|c||c|c||c|c|c|}
			\hline
			\multirow{2}*{}& \multicolumn{2}{c||}{$\gamma_h\sim 10^{-7}$} &\multicolumn{2}{c||}{$\gamma_h\sim Re^{1/3}h^{2/3}$}
			&\multicolumn{2}{c||}{$\gamma_h\sim Re^{1/2}h$} &\multicolumn{2}{c|}{$\gamma_h\sim Re\,h^2$} \\
			\cline{1-9}
			h & err & rate & err & rate & err & rate & err & rate\\
			\hline
			\hline
			1/20  & 7.49E+01 &       & 1.33E-02 &      & 1.49E-02 &      & 9.23E-02 &       \\ \hline
			1/40  & 4.44E+01 & 0.75  & 1.05E-02 & 0.34 & 1.48E-02 & 0.00 & 1.88E-01 & -1.03 \\ \hline
			1/80  & 3.58E+01 & 0.31  & 8.32E-03 & 0.33 & 1.53E-02 & -0.05 & 3.79E-01 & -1.01 \\ \hline
			1/160 & 1.43E+01 & 1.32  & 6.62E-03 & 0.33 & 1.58E-02 & -0.04 & 7.59E-01 & -1.00 \\ \hline
			1/320 &  &  & 5.26E-03 & 0.33 & 1.60E-02 & -0.02 &  &  \\ \hline
		\end{tabular}
	\end{center}
	\caption{Case \texttt{MP-NC} : $\lVert \vec p-\vec p_h \rVert_{L^2}$ for $P_1$-$P_0$-element}\label{tab:PL2P1P0_NC}
\end{table}

As a result, even though ($P_1$-$P_1$) elements give better results than expected, the numerical results are in accordance with our error estimates. These estimates seems to be sharp since the predicted rates are recovered for ($P_1$-$P_0$) elements.

\subsection{Natural convection in 2D}

The convergence of our operator being assessed, we can illustrate the convergence properties for the natural convection problem.  Note that in numerical simulations, we used a second order discretization in time, which only improves the accuracy with respect to $\delta t$. The remaining of the estimate \eqref{eq-improved} still holds.

As for the modified projection, we illustrate Corollary \ref{eq-improved} on two examples. The first one consists of a manufactured solution, see \cite{nourgaliev2016fully}. The second example is a physical case of natural convection, compared to the reference solution described in Section~\ref{Steady natural convection}. A second validation is obtained by comparing the vertical velocity profile at mid-domain ($y=0.5$) to a spectral approximation~\cite{LeQuere91}.


\begin{table}[H]\begin{center}
		\begin{tabular}{|c|c|c|}
			\hline
			Leading terms	& $\gamma_h$ & expected order \\ \hline\hline
			\multirow{4}{*}{$\delta t, h^2, \frac{h^{4-\sigma}}{\gamma_h^{2-\sigma}},\gamma_h^2$ } & $10^{-7}$ & Stability \\ \cline{2-3} 
			& $Re^{1/3}h^{2/3}$ & Convergence with an order 2/3 \\ \cline{2-3} 
			& $Re^{1/2}h$ & Convergence with an order 1 \\ \cline{2-3} 
			& $Re\,h^2$ & Stability \\ \hline
		\end{tabular}
	\end{center}
	\caption{Expected order of convergence for $\lVert u-u_h\rVert_{L^2}$ and $\lVert \theta-\theta_h\rVert_{L^2}$}
	\label{tab:L2-norm-U-expected-order-NC}
\end{table}

\subsubsection{Manufactured solution  (case \texttt{NC-Nour})}

We first consider the manufactured time-dependent solution suggested in \cite{nourgaliev2016fully}:
\begin{align}
\label{eq-manufN}
u_1(x,y,t) &= \left( \delta U_0 + \alpha_u \, \sin(t) \right) \, \cos(x+ \gamma_1 t) \, \sin(y+ \gamma_2 t), \\ \nonumber
u_2(x,y,t) &= - \left( \delta U_0 + \alpha_u \sin(t) \right) \, \sin(x+ \gamma_1 t) \, \cos(y+ \gamma_2 t), \\ \nonumber
\theta(x,y,t) &= \bar{\theta} + \left( \delta \theta_0 + \alpha_t \sin(t) \right) \, \cos(x+ \gamma_1 t) \, \sin(y+ \gamma_2 t), \\ \nonumber
p(x,y,t) &= \bar{P} + \left(\delta P_0 + \alpha_p \sin(t) \right) \, \sin(x+ \gamma_1 t) \, \cos(y+ \gamma_2 t), 
\end{align}
The values of the constants are reported in Table \ref{tab-constant}.
\begin{table}[t]
	\centering
	\begin{tabular}{*{10}{c}}
		$\gamma_1$ & $\gamma_2$ & $\bar{P}$ & $\bar{\theta}$ & $\delta P_0$ & $\delta \theta_0$ & $\delta U_0$ & $\alpha_p$ & $\alpha_u$ & $\alpha_t$\\
		\midrule
		$0.1$ & $0.1$ & $0$ & $1.0$ &  $0.1$ & $1.0$ & $1.0$ & $0.05$ & $0.4$ & $0.1$ \\
	\end{tabular}
	\caption{Parameters for case \texttt{NC-Nour} \eqref{eq-manufN}.}
	\label{tab-constant}
\end{table}
The corresponding source terms are:
\begin{eqnarray}
f_{u_1} &=& \alpha_u \, \cos(t) \, \cos(a) \sin(b) - U_c \, \gamma_1 \, \sin(a) \sin(b) + U_c \, \gamma_2  \, \cos(a)\cos(b) \\ \nonumber
& & - U_c \,  u_1(x,y,t) \, \sin(a) \sin(b) + U_c \,  u_2(x,y,t) \, \cos(a) \cos(b)
+ P_c \, \cos(a) \cos(b)\\ \nonumber
& & + \frac{2}{Re} \, u_1(x,y,t), \\	  \nonumber
f_{u_2} &=& - \alpha_u \,  \cos(t)  \, \sin(a) \cos(b) - U_c \,  \gamma_1  \,  \cos(a) \cos(b) + U_c \,  \gamma_2 \,  \sin(a)\sin(b) \\ \nonumber
& & - U_c \,  u_1(x,y,t)  \,  \cos(a) \cos(b) + U_c  \, u_2(x,y,t)  \,  \sin(a) \sin(b)
-  P_c  \,  \sin(a)  \,  \sin(b)\\ \nonumber
& &+ \frac{2}{Re} \,  u_2(x,y,t)
- \frac{Ra}{Pr Re^2} \,  T(x,y,t), \\  \nonumber
f_{\theta} &=& \alpha_t \,  \cos(t) \,  \cos(a) \sin(b) -  \theta_c  \,  \gamma_1 \,  \sin(a) \sin(b) + \theta_c \,   \gamma_2  \,  \cos(a)\cos(b) \\ \nonumber
& &-  \theta_c \,  u_1(x,y,t)  \,  \sin(a) \sin(b)  
+   \theta_c  \,  u_2(x,y,t)  \, \cos(a) \cos(b) 
+ \frac{2 K}{Re Pr} \,  \theta_c  \, \cos(a) \sin(b), \nonumber
\end{eqnarray}
where $a = (x+ \gamma_1 t), \,
b = (y+ \gamma_2 t)$ and  
$U_c = (\delta U_0 + \alpha_u \sin(t)), \,
\theta_c = (\delta \theta_0 + \alpha_u \sin(t)), \,
P_c = (\delta P_0 + \alpha_u \sin(t))$.

We want to assess the accuracy of our estimates with respect to $h$. Since we use a second order scheme in time, we choose $\delta t\simeq h$. Dimensionless parameters are chosen to emulate the convection of air, with a Rayleigh number $Ra = 10^6$, a Prandtl number $Pr = 0.71$ and Reynolds number $Re=\sqrt{Ra/Pr}$. The errors are computed at the final time $t_f=\pi/2$. 


\begin{table}[H]
	\begin{center}
		\begin{tabular}{|c||c|c||c|c||c|c||c|c|}
			\hline
			\multirow{2}*{}& \multicolumn{2}{c||}{$\gamma_h\sim 10^{-7}$} &\multicolumn{2}{c||}{$\gamma_h\sim Re^{1/3}h^{2/3}$}  &\multicolumn{2}{c||}{$\gamma_h\sim Re^{1/2}h$} &\multicolumn{2}{c|}{$\gamma_h\sim Re\,h^2$} \\
			\cline{1-9}
			h & \text{err} & \text{rate} & \text{err} & \text{rate} & \text{err} & \text{rate} & \text{err} & \text{rate} \\
			\hline
			\hline
			1/20  & 8.79E-04 &      & 1.49E-02 &      & 1.71E-02 &      & 2.48E-02 &      \\ \hline
			1/40  & 2.15E-04 & 2.03 & 9.89E-03 & 0.59 & 9.48E-03 & 0.85 & 8.31E-03 & 1.58 \\ \hline
			1/80  & 5.25E-05 & 2.03 & 6.48E-03 & 0.61 & 4.99E-03 & 0.93 & 2.22E-03 & 1.90 \\ \hline
			1/160 & 1.30E-05 & 2.01 & 4.19E-03 & 0.63 & 2.56E-03 & 0.96 & 5.64E-04 & 1.98 \\ \hline
			1/320 &          &      & 2.69E-03 & 0.64 & 1.30E-03 & 0.98 & 1.42E-04 & 1.99 \\ \hline
		\end{tabular}
	\end{center}
	\caption{Case \texttt{NC-Nour} : $\lVert \vec u-\vec u_h \rVert_{L^2}$ for $P_1$-$P_1$-$P_1$ element}
	\label{tab:Nourgaliev-UL2P1P1}
\end{table}

\begin{table}[H]
	\begin{center}
		\begin{tabular}{|c||c|c||c|c||c|c||c|c|}
			\hline
			\multirow{2}*{}& \multicolumn{2}{c||}{$\gamma_h\sim 10^{-7}$} &\multicolumn{2}{c||}{$\gamma_h\sim Re^{1/3}h^{2/3}$}  &\multicolumn{2}{c||}{$\gamma_h\sim Re^{1/2}h$} &\multicolumn{2}{c|}{$\gamma_h\sim Re\,h^2$} \\
			\cline{1-9}
			h & \text{err} & \text{rate} & \text{err} & \text{rate} & \text{err} & \text{rate} & \text{err} & \text{rate} \\		\hline\hline
			1/20  & 1.38E-03 &      & 4.57E-03 &      & 5.25E-03 &      & 8.16E-03 &      \\ \hline
			1/40  & 3.37E-04 & 2.03 & 2.41E-03 & 0.92 & 2.31E-03 & 1.19 & 2.03E-03 & 2.01 \\ \hline
			1/80  & 8.47E-05 & 1.99 & 1.47E-03 & 0.71 & 1.14E-03 & 1.02 & 5.31E-04 & 1.94 \\ \hline
			1/160 & 2.12E-05 & 2.00 & 9.34E-04 & 0.66 & 5.75E-04 & 0.99 & 1.35E-04 & 1.98 \\ \hline
			1/320 &          &      & 5.93E-04 & 0.65 & 2.88E-04 & 0.99 & 3.38E-05 & 2.00 \\ \hline
		\end{tabular}
		\caption{Case \texttt{NC-Nour} : $\lVert \theta-\theta_h \rVert_{L^2}$ for $P_1$-$P_1$-$P_1$ element}
		\label{tab:Nourgaliev-TL2P1P1}
	\end{center}
\end{table}

\begin{table}[H]
	\begin{center}
		\begin{tabular}{|c||c|c||c|c||c|c||c|c|}
			\hline
			\multirow{2}*{}& \multicolumn{2}{c||}{$\gamma_h\sim 10^{-7}$} &\multicolumn{2}{c||}{$\gamma_h\sim Re^{1/3}h^{2/3}$}  &\multicolumn{2}{c||}{$\gamma_h\sim Re^{1/2}h$} &\multicolumn{2}{c|}{$\gamma_h\sim Re\,h^2$} \\
			\cline{1-9}
			h & \text{err} & \text{rate} & \text{err} & \text{rate} & \text{err} & \text{rate} & \text{err} & 	\text{rate} \\ 	\hline\hline
			1/20  & 1.51E+04 &      & 6.81E-03 &      & 8.07E-03 &      & 1.44E-02 &      \\ \hline
			1/40  & 3.81E+03 & 1.99 & 4.71E-03 & 0.53 & 4.51E-03 & 0.84 & 3.96E-03 & 1.87 \\ \hline
			1/80  & 9.52E+02 & 2.00 & 2.81E-03 & 0.74 & 2.17E-03 & 1.06 & 1.09E-03 & 1.86 \\ \hline
			1/160 & 2.38E+02 & 2.00 & 1.70E-03 & 0.73 & 1.04E-03 & 1.06 & 5.27E-04 & 1.05 \\ \hline
			1/320 &          &      & 1.05E-03 & 0.70 & 5.11E-04 & 1.03 & 5.25E-04 & 0.01 \\ \hline
		\end{tabular}
		\caption{Case \texttt{NC-Nour} : $\lVert \vec p-\vec p_h \rVert_{L^2}$ for $P_1$-$P_1$-$P_1$ element}
		\label{tab:Nourgaliev-PL2P1P1}
	\end{center}
\end{table}

Since the error estimate of natural convection involves the error estimates of modified projection, we observe the same behaviour as before i.e. the ($P_1$-$P_1$) finite elements seem to yield a better convergence rate than expected. However, these results are in good agreement with our estimates. In particular, the correct convergence rates for the velocity and temperature are observed for $\gamma_h=Re^{1/3}h^{2/3}$ and $\gamma_h=Re^{1/2}h$ and the case $\gamma_h=10^{-7}$ yields computational issues for small values of $h$. 

\subsubsection{Natural convection in a square (case \texttt{NC-Sq})}\label{NC-Numerical-Comparison}

We are now interested in a physical case. We consider here the reference solution described in Section~\ref{Steady natural convection}. We numerically solve equations describing the classical natural convection with a finite element scheme ($P_1$-$P_1$-$P_1$) and we compare our results to this reference solution. The Rayleigh number is taken as $Ra=10^4$, the Prandtl number $Pr=0.71$ and the Reynolds number is $Re=\sqrt{Ra/Pr}$. The results are presented in Tables~\ref{tab:NC-Sq_u}-\ref{tab:NC-Sq_theta}-\ref{tab:NC-Sq_p}.

\begin{table}[H]
	\begin{center}
		\begin{tabular}{|c||c|c||c|c||c|c||c|c|}
			\hline
			\multirow{2}*{}&\multicolumn{2}{c||}{$\gamma_h\sim 10^{-7}$} &\multicolumn{2}{c||}{$\gamma_h\sim Re^{1/3}h^{2/3}$}
			&\multicolumn{2}{c||}{$\gamma_h\sim Re^{1/2}h$} &\multicolumn{2}{c|}{$\gamma_h\sim Re\,h^2$} \\
			\cline{1-9}
			h  & \text{err} & \text{rate} & \text{err} & \text{rate} & \text{err} & \text{rate} & \text{err} & \text{rate}\\
			\hline
			\hline
			1/20  & 5.57E-03 &      & 1.15E-02 &      & 6.08E-03 &      & 5.28E-03 &      \\ \hline
			1/40  & 1.39E-03 & 2.00 & 7.00E-03 & 0.71 & 2.34E-03 & 1.38 & 1.34E-03 & 1.98 \\ \hline
			1/80  & 3.47E-04 & 2.00 & 4.43E-03 & 0.66 & 1.08E-03 & 1.12 & 3.36E-04 & 1.99 \\ \hline
			1/160 & 8.68E-05 & 2.00 & 2.81E-03 & 0.66 & 5.30E-04 & 1.03 & 8.50E-05 & 1.98 \\ \hline
			1/320 & 3.23E-05 & 1.43 & 8.09E-04 & 1.80 & 2.65E-04 & 1.00 & 3.21E-05 & 1.40 \\ \hline
		\end{tabular}
	\end{center}
	\caption{Case \texttt{NC-Sq} : $\lVert \vec u-\vec u_h \rVert_{L^2}$ for $P_1$-$P_1$-$P_1$ elements}\label{tab:NC-Sq_u}
\end{table}

\begin{table}[H]
	\begin{center}
		\begin{tabular}{|c||c|c||c|c||c|c||c|c|}
			\hline
			\multirow{2}*{}&\multicolumn{2}{c||}{$\gamma_h\sim 10^{-7}$} 	&\multicolumn{2}{c||}{$\gamma_h\sim Re^{1/3}h^{2/3}$}
			&\multicolumn{2}{c||}{$\gamma_h\sim Re^{1/2}h$} &\multicolumn{2}{c|}{$\gamma_h\sim Re\,h^2$} \\
			\cline{1-9}
			h  & \text{err} & \text{rate} & \text{err} & \text{rate} & \text{err} & \text{rate} & \text{err} & \text{rate}\\
			\hline
			\hline
			1/20  & 4.33E-03 &      & 1.76E-02 &      & 8.86E-03 &      & 6.59E-03 &      \\ \hline
			1/40  & 1.08E-03 & 2.01 & 1.04E-02 & 0.76 & 3.65E-03 & 1.28 & 1.67E-03 & 1.98 \\ \hline
			1/80  & 2.74E-04 & 1.97 & 6.42E-03 & 0.69 & 1.66E-03 & 1.14 & 4.24E-04 & 1.98 \\ \hline
			1/160 & 1.00E-04 & 1.45 & 4.06E-03 & 0.66 & 8.21E-04 & 1.01 & 1.41E-04 & 1.58 \\ \hline
			1/320 & 5.54E-05 & 0.86 & 1.20E-03 & 1.76 & 4.25E-04 & 0.95 & 6.70E-05 & 1.08 \\ \hline
		\end{tabular}
		\caption{Case \texttt{NC-Sq} : $\lVert \theta-\theta_h \rVert_{L^2}$ for $P_1$-$P_1$-$P_1$ elements}\label{tab:NC-Sq_theta}
	\end{center}
\end{table}

\begin{table}[H]
	\begin{center}
		\begin{tabular}{|c||c|c||c|c||c|c||c|c|}
			\hline
			\multirow{2}*{}&\multicolumn{2}{c||}{$\gamma_h\sim 10^{-7}$} 	&\multicolumn{2}{c||}{$\gamma_h\sim Re^{1/3}h^{2/3}$}
			&\multicolumn{2}{c||}{$\gamma_h\sim Re^{1/2}h$} &\multicolumn{2}{c|}{$\gamma_h\sim Re\,h^2$} \\
			\cline{1-9}
			h  & \text{err} & \text{rate} & \text{err} & \text{rate} & \text{err} & \text{rate} & \text{err} & 	\text{rate}\\
			\hline
			\hline
			\hline
			1/20  & 2.34E-02 &      & 7.73E-03 &      & 1.15E-02 &      & 1.46E-02 &      \\ \hline
			1/40  & 2.17E-02 & 0.11 & 4.05E-03 & 0.93 & 7.22E-03 & 0.67 & 1.36E-02 & 0.10 \\ \hline
			1/80  & 2.10E-02 & 0.04 & 2.13E-03 & 0.93 & 4.32E-03 & 0.74 & 1.33E-02 & 0.04 \\ \hline
			1/160 & 2.08E-02 & 0.02 & 1.15E-03 & 0.89 & 2.51E-03 & 0.78 & 1.31E-02 & 0.01 \\ \hline
			1/320 & 2.07E-02 & 0.01 & 7.13E-04 & 0.69 & 1.46E-03 & 0.79 & 1.31E-02 & 0.01 \\ \hline
		\end{tabular}
	\end{center}
	\caption{Case \texttt{NC-Sq} : $\lVert \vec p-\vec p_h \rVert_{L^2}$ for $P_1$-$P_1$-$P_1$ elements}\label{tab:NC-Sq_p}
\end{table}

Once again, the results are in accordance with the theoretical results, even though the case $\gamma_h=Re\; h^2$ exhibits slightly better results than expected.

We strengthen our code validation by comparing the final state with a profile obtained by a spectral code \cite{LeQuere91}. The vertical $v_{FE}$ velocity profile is extracted at mid-domain ($y=0.5$) and plotted in Figure \ref{NC_Comparison_LeQuere} for the pair ($P_1$-$P_1$) and the different values of $\gamma_h$ and compared to the reference solution $v_{LQ}$. The difference $\|v_{FE}-v_{LQ}\|_{L^2}$ is reported in table~\ref{Tab:ErrL2_LEQUERE}.

\begin{figure}
		\centering
		\includegraphics[width=0.7\textwidth]{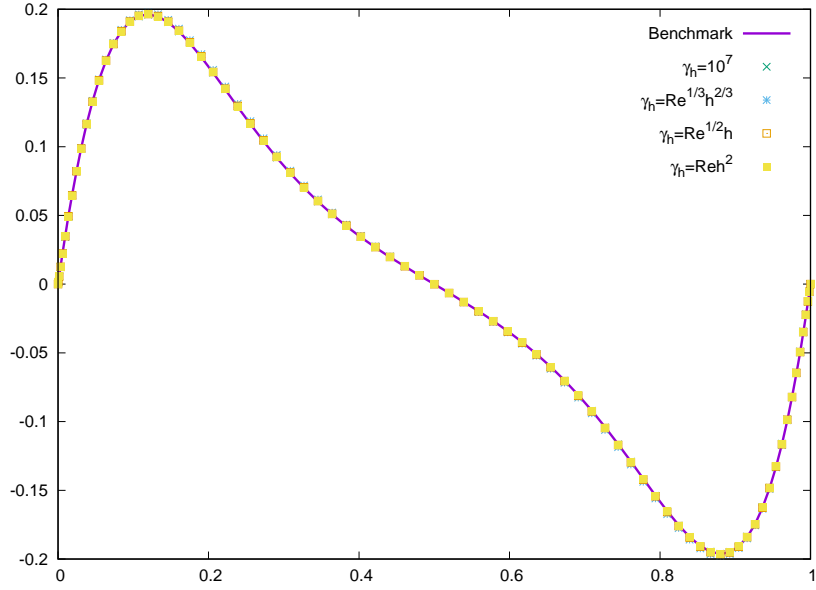} \hfill
\caption{Case \texttt{NC-Sq}: vertical velocity profile at $y=0.5$; comparison with results from \cite{LeQuere91}. The mesh size is $h=1/80$.}
	\label{NC_Comparison_LeQuere}
\end{figure}

\begin{table}[H]
	\begin{center}
		\begin{tabular}{|c||c|c||c|c||c|c||c|c|}
			\hline
			\multirow{2}*{}& \multicolumn{2}{c||}{$\gamma_h=10^{-7}$} &\multicolumn{2}{c||}{$\gamma_h=Re^{1/3} h^{2/3}$} &\multicolumn{2}{c||}{$\gamma_h=Re^{1/2}h$} &\multicolumn{2}{c|}{$\gamma_h=Re\,h^2$} \\
			\cline{1-9}
			h & \text{err} & \text{rate} & \text{err} & \text{rate} & \text{err} & \text{rate} & \text{err} & \text{rate} \\
			\hline
			\hline
			1/20  & 5.08E-03 &      & 7.76E-03 &      & 5.14E-03 &      & 4.85E-03 &      \\ \hline
			1/40  & 6.01E-04 & 3.08 & 1.55E-03 & 2.32 & 6.32E-04 & 3.02 & 5.57E-04 & 3.12 \\ \hline
			1/80  & 7.09E-05 & 3.08 & 4.29E-04 & 1.85 & 8.30E-05 & 2.93 & 6.43E-05 & 3.11 \\ \hline
			1/160 & 7.36E-06 & 3.27 & 1.42E-04 & 1.60 & 1.09E-05 & 2.93 & 6.59E-06 & 3.29 \\ \hline
			1/320 & 7.34E-07 & 3.33 & 1.01E-05 & 3.82 & 1.24E-06 & 3.13 & 6.29E-07 & 3.39 \\ \hline
		\end{tabular}
		\caption{Case \texttt{NC-Sq}: error on the centerline velocity $\lVert v_{FE}-v_{LQ} \rVert_{L^2}$.}
		\label{Tab:ErrL2_LEQUERE}
	\end{center}
\end{table}

Differences decrease when the mesh resolution is increased. Our comparison with the spectral code guarantees that our code gives a correct final solution. The numerical results for these different cases demonstrate the validity of Corollary \eqref{eq-improved}. The numerical results we presented confirm that the choice of $\gamma_h=Re^{1/3}h^{2/3}$ is suitable to obtain convergence on both velocity and pressure, but with non optimal rate $2/3$ for the velocity. If one is interested only in the velocity accuracy, then the choice $\gamma_h=Re^{1/2}h$ yields the best convergence rate on the velocity, while loosing convergence on the pressure. In summary, these two choices along with $P_1$-$P_1$ finite elements are suitable for simulating natural convection problems.

\subsection{Computational cost}

\subsubsection{2D-Natural Convection}

As stated before, the non linear system \eqref{eq-div-discrete}-\eqref{eq-qmvt-discrete}-\eqref{eq-energ-discrete} is solved using a Newton algorithm. The number of Newton iterations does not change when changing the order of the polynomials. Thus, using a direct solver for the linear systems that are involved ensures a computational gain for the linear element compared to the quadratic element. However, it is worth mentioning that, the use of an iterative solver (e.g. GMRES as in our simulations) yields a reduction of the computational cost, even though the resulting systems are more ill-conditioned.

We compare the computation time between ($P_1$-$P_1$-$P_1$) and ($P_2$-$P_1$-$P_2$) finite elements for some of the previously tested cases. In parallel, we compare the number of degrees of freedom for both pairs related to the mesh size $h$. Since we observe similar behaviours for different values of $h$, we report in Table~\ref{Tab:CPU_NC} the values for $h=1/160$. The number of degrees of freedom for the ($P_2$-$P_1$-$P_2$) is about three times that of ($P_1$-$P_1$-$P_1$) elements.

\begin{table}[H]
	\begin{center}
		\begin{tabular}{c|c|c|c||c|c|c|}
			\hline
			\multicolumn{1}{|c||}{}        & \multicolumn{3}{c||}{Case \texttt{NC-Nour}}  & \multicolumn{3}{c|}{Case \texttt{NC-Sq}}   \\ \hline\hline
			\multicolumn{1}{|c||}{$\gamma_h$}        &$ P_1$-$P_1$-$P_1$  & $P_2$-$P_1$-$P_2$ & rate  & $P_1$-$P_1$-$P_1$  & $P_2$-$P_1$-$P_2$  & rate  \\ \hline\hline
			\multicolumn{1}{|c||}{$10^{-7}$}         &   40 683 & 8 641  & 0.21   & 84 282 & 16 607 & 0.20   \\ \hline
			\multicolumn{1}{|c||}{$Re^{1/3}h^{2/3}$} & 2 465  & 8 371  & 3.40 & 3 432 & 12 438 & 3.62    \\ \hline
			\multicolumn{1}{|c||}{$Re^{1/2}h $}     & 2 419  & 8 424  & 3.48   & 4 499 & 12 714 & 2.83  \\ \hline
			\multicolumn{1}{|c||}{$Re\,h^2$}           &  2 470  & 8 575  & 3.47 & 13 951 & 14 591 & 1.05  \\ \hline\hline
			\multicolumn{1}{|c||}{ndof}              &  103 684 & 335 044  & 3.23 &  103 684 & 335 044  & 3.23 \\ \hline
		\end{tabular}
		\caption{CPU-time (s) for different choices of $\gamma_h$, with $h=1/160$.}
		\label{Tab:CPU_NC}
	\end{center}
\end{table}

It is interesting to note that the choice of $\gamma_h$ with Taylor-Hood elements ($P_2$-$P_1$) for velocity and pressure, and $P_2$ for temperature does not affect the computational time, while it has a visible impact for the ($P_1$-$P_1$-$P_1$) computations. The choices $\gamma_h\sim Re^{1/3}h^{2/3}$ and $\gamma_h\sim Re^{1/2}h$, for which we have proven convergence of the scheme, give the smallest computation times. This choices also bring a noticeable gain compared to $P_2$ computations.

\subsubsection{3D Natural Convection}

We now consider the same case of natural convection of air described in \ref{Steady natural convection} and add the third dimension in space. We thus simulate the thermally driven cubic cavity $[0,1]^3$, filled with air. The temperature is fixed on the left (hot) wall surface and the right (cold) wall surface. All the other lateral surfaces are adiabatic. No-slip walls are applied for the velocity on all boundary surfaces. Simulations are made for $Ra=10^4$ on a fixed and uniform mesh with $h=1/40$ for the finite element pairs ($P_1$-$P_1$-$P_1$) and ($P_2$-$P_1$-$P_2$). CPU-time for different values of $\gamma_h$ is given in Table~\ref{Tab:CPU_Convection_Naturelle_3D}.

\begin{table}[H]
	\begin{center}
		\begin{tabular}{c|c|c|c|}
			\hline
			\multicolumn{1}{|c||}{$\gamma_h$}        & $P_1$-$P_1$-$P_1$  & $P_2$-$P_1$-$P_2$  & rate  \\ \hline\hline
			\multicolumn{1}{|c||}{$10^{-7}$}         & 7 761   & 20 059     & 2.62  \\ \hline
			\multicolumn{1}{|c||}{$Re^{1/3}h^{2/3}$} & 5 476  & 17699   & 3.23  \\ \hline
			\multicolumn{1}{|c||}{$Re^{1/2}h $}     & 5 713   & 17 964   & 3.14  \\ \hline
			\multicolumn{1}{|c||}{$Re\,h^2$}           & 6 043   & 19 084   & 3.16  \\ \hline\hline
			\multicolumn{1}{|c||}{ndof}              & 344 605    & 2 194 685   & 6.37  \\ \hline
		\end{tabular}
		\caption{CPU-time for 3D-Natural-Convection.}
		\label{Tab:CPU_Convection_Naturelle_3D}
	\end{center}
\end{table}
As for the 2D case, we observe a non negligible reduction in the computational time, which supports the use of low-order finite elements.

\section{Conclusion}\label{sec-conclusion}

We introduced a new projection operator to establish an error estimate for the approximation of the natural convection problem with linear elements. Using this operator and choosing a stabilization $\gamma_h$ accordingly, we are able to recover almost first order accuracy in space. The error estimates on the projection and for the natural convection problem are validated with extensive examples, which indicate that the predicted convergence rate is optimal. Moreover, the choice of the stabilization parameter can be done in two ways. Using $\gamma_h\sim Re^{1/2}h$  gives a better convergence for the velocity, but does not yield any convergence for the pressure. Using $\gamma_h\sim Re^{1/3}h^{2/3}$ would allow a slow convergence in pressure, even though it reduces the accuracy for the velocity. In both cases, the computational time is reduced, when compared to the standard Taylor-Hood elements for the fluid equations. This makes the suggested scheme tractable for applications, even for 3D problems.

Preliminary simulations (not shown here) suggest that low order finite element approximation might also be used for more complex applications, such as simulation of PCM (e.g. cases presented in \cite{ArticleToolBox,ParallelToolBox}). This remains to be proven using similar techniques : however, the analysis of convergence is more involved, due to the phase change which adds non linear terms in the energy equation.

\bibliographystyle{elsarticle-num}
\bibliography{bib_P1_P1.bib}

\end{document}